\documentclass{article}
\usepackage[utf8]{inputenc}

\usepackage{float}
\usepackage{graphicx}
\usepackage{amsmath}
\usepackage{enumitem}
\usepackage{amscd}
\usepackage{amsthm}
\usepackage{mathtools}
\usepackage{amssymb}
\usepackage{tikz-cd}
\usepackage{gensymb}
\usepackage{caption}
\usepackage{dsfont}
\usepackage[font=small,labelfont=bf,width=1\textwidth]{caption}
\usepackage{multicol}
\usepackage{subfig}
\usepackage[left=3.5cm,top=4cm,right=3.5cm,bottom=4cm]{geometry}

\usepackage{siunitx}
\usepackage{lipsum}
\usepackage{xfrac}
\usepackage{marvosym}
\usepackage{epic}
\usepackage{wasysym}
\usepackage{epstopdf}
 
\usepackage{hyperref}
\usepackage{wrapfig}
\usepackage{enumitem}
\usepackage{enumerate}
\usepackage{color}

\theoremstyle{plain}
\newtheorem{theorem}{Theorem}[section]
\newtheorem{lemma}[theorem]{Lemma}
\newtheorem{corollary}[theorem]{Corollary}

\newtheorem{prop}[theorem]{Proposition}
\newtheorem{remark}[theorem]{Remark}

\theoremstyle{definition}
\newtheorem{definition}[theorem]{Definition}

\mathtoolsset{showonlyrefs}

\newcommand{\cj}{\bar}
\newcommand{\cw}{\bar{w}}

\title{Brownian motion and stochastic areas on complex full flag manifolds }

\author{Fabrice Baudoin\footnote{Research partially supported by grant 10.46540/4283-00175B from Independent Research Fund Denmark}, Nizar Demni, Teije Kuijper, Jing Wang\footnote{Research was supported in part by NSF Grant DMS-2246817}}
\begin{document}

\maketitle

\begin{abstract}
We show that the Brownian motion on the complex full flag manifold can be represented by a matrix-valued diffusion obtained from the unitary Brownian motion. This representation actually leads to an explicit formula for the characteristic function of the  joint distribution of the  stochastic areas on the full flag manifold.  The limit law for those stochastic areas is shown to be a multivariate Cauchy distribution with independent and identically distributed entries. Using a deep connection between area functionals on the flag manifold and winding functionals on complex spheres, we establish new results about simultaneous Brownian windings on the complex sphere and their asymptotics. As a byproduct, our work also unveils a new probabilistic interpretation of the Jacobi operators and polynomials on simplices.\end{abstract}

\tableofcontents

\newpage

\section{Introduction}

The study of stochastic area functionals has deep roots in both probability and geometry, tracing back to Paul L\'evy’s foundational work \cite{Levy} on planar Brownian motion. Over the years, this subject has evolved into a rich and extensive theory; see \cite{book} for a recent survey. One of the main motivations of this paper is to extend this theory to the setting of flag manifolds.

Flag manifolds play key roles in differential geometry \cite{MR2264399}, representation theory \cite{MR1491979}, algebraic geometry \cite{Brion2005}, physics \cite{MR4117041}, and numerical analysis \cite{MR4445465}. Among them, the complex full flag manifold $F_{1,2,\ldots,n-1}(\mathbb{C}^n)$ parametrizes nested sequences of complex subspaces:  
\[
\{0\} \subsetneq W_1 \subsetneq \cdots \subsetneq W_{n-1} \subsetneq \mathbb{C}^n.
\]  
The importance of this space stems from the fact that any complex partial flag manifold can be obtained from it via a submersion \cite[\textsection 1.2]{Brion2005}. It also admits a homogeneous Riemannian structure as $\mathbf{U}(n)/\mathbf{U}(1)^n$, where $\mathbf{U}(n)$ is the unitary group and $\mathbf{U}(1)^n$ its maximal torus of diagonal unitary matrices. Though not a symmetric space, the full flag manifold has a complex K\"ahler structure \cite{MR77878}, which will be a crucial ingredient in our investigations of: the Brownian motion on $F_{1,2,\ldots,n-1}(\mathbb{C}^n)$; its associated stochastic area functionals; and its connections to Brownian winding on complex spheres.

Brownian motion on Lie groups and homogeneous spaces has long been a cornerstone of stochastic differential geometry \cite{MR1284654,Dyn61,MR0359023}. A key contribution of this paper is the construction of Brownian motion on the full flag manifold via projection from unitary Brownian motion. Expressing this process in local affine coordinates derived from the quotient structure allows us to explicitly compute its generator, which governs the radial dynamics of the process. In
particular, these dynamics will be identified with Jacobi diffusions on simplices.

Building on the general theory described in \cite{book}, we then use the K\"ahler structure of $\mathbf{U}(n)/\mathbf{U}(1)^n$ to define a natural stochastic area process. Specifically, the Riemannian fibration  
\[
\mathbf{U}(1)^n \to \mathbf{U}(n) \to \mathbf{U}(n)/\mathbf{U}(1)^n
\]
allows us to view $\mathbf{U}(n)$ as a torus bundle over the full flag manifold, leading to an intrinsic $n$-dimensional area process. For a horizontal Brownian motion on $\mathbf{U}(n)$, these stochastic area measures accumulated phase differences across the torus fibers. Using a skew-product decomposition, we derive the joint characteristic function of these areas and establish that, as $t\to +\infty$, their limiting distribution follows a multivariate Cauchy law with independent components after proper rescaling. This limit theorem relies on spectral properties of Jacobi operators and their associated orthogonal polynomials on simplices.

A striking application arises in the study of simultaneous Brownian windings on the complex sphere $\mathbb{S}^{2n-1} \subset \mathbb{C}^n$, generalizing a result in \cite{baudoin2024lawindexbrownianloops}. By linking the stochastic areas of the flag manifold to angular windings of spherical Brownian motion, we establish asymptotic independence of the winding processes. Specifically, upon proper rescaling, these windings also converge in distribution to independent Cauchy random variables.

The paper is organized as follows: Section \ref{sec:preliminaries} reviews complex flag manifolds and Jacobi polynomials on simplices; Section \ref{sec:BM-on-FF} constructs the Brownian motion on the full flag manifold as a suitable projection of the unitary Brownian motion and analyzes its radial dynamics; Section \ref{sec:SAF-and-skew-product} introduces stochastic area functionals, derives the corresponding characteristic functions, and contains a proof of convergence of said functionals to a multivariate Cauchy distribution; and finally, Section \ref{sec:Brownian-winding-on-spheres} applies these results to simultaneous Brownian windings on complex spheres, establishing asymptotic laws and connections to Euclidean Brownian motion.

\section{Preliminaries}\label{sec:preliminaries}
This section is devoted to recalling some geometric and analytic preliminaries that will be used in the subsequent analysis. We begin by reviewing basic geometric facts about the flag manifold and the Riemannian submersion structure it carries. We then turn to a brief overview of the Jacobi operator and Jacobi polynomials on simplices, which will play a key role in the later sections. Throughout the paper, let $n \ge 2$ be an integer.

\subsection{Complex full flag manifolds}\label{sec-full-flag}
In this section, we review the geometric structure of complex full flag manifolds, including their realization as homogeneous spaces and the associated Riemannian submersion structure.

\begin{definition}\label{def-flag}
A \emph{flag} $(W_1,\dots,W_{k})$ of $\mathbb C^n$ is a sequence
\begin{align*}
    \{ 0 \} =:W_0\subsetneq W_1\subsetneq\dots\subsetneq W_{k}\subsetneq W_{k+1} := \mathbb C^n
\end{align*}
of complex subspaces of $\mathbb C^n$.
The number $k$ is called the \emph{length} of the flag and the $k$-tuple
\begin{align*}
    (\dim (W_1),\dots ,\dim (W_k))
\end{align*}
the \emph{signature} of the flag.
\end{definition}

\begin{definition}
    Let $V$ be a complex vector space of dimension $n$.
    The \emph{(complex) partial flag manifold} $F_{d_1,\dots ,d_k}(\mathbb C^n)$ \emph{of signature $(d_1,\dots ,d_k)$} is the collection of flags of $\mathbb C^n$ of signature $(d_1,\dots ,d_k)$. The \emph{(complex) full flag manifold} is the flag manifold $F_{1,2,\dots ,n-1}(\mathbb C^n)$.
\end{definition}

To study and define the smooth structure on $F_{1,2,\dots ,n-1}(\mathbb C^n)$ we will make a first identification.
Note that the subgroup of invertible upper triangular matrices $T^n(\mathbb{C})$ acts transitively on the set $\mathrm{GL}_n(\mathbb C)$ of $n \times n$ invertible matrices by right multiplication and that we can identify $F_{1,2,\dots ,n-1}(\mathbb C^n)$ as the quotient space $\mathrm{GL}_n(\mathbb C) /T^n(\mathbb{C})$.
Indeed, consider the canonical basis  $e_1,\dots,e_n$ of $\mathbb{C}^n$, then the surjective map 
\begin{align*}
\begin{array}{ccc}
   \mathrm{GL}_n(\mathbb C)&  \to &  F_{1,2,\dots ,n-1}(\mathbb C^n)\\
   M & \mapsto &  (W_1,\dots, W_{n-1})
\end{array}
\end{align*}
with $W_i=\mathrm{span}(Me_1,Me_2,\dots,Me_i)$ is invariant by this action and thus  descends into a bijection $\mathrm{GL}_n(\mathbb C) /T^n(\mathbb{C}) \to F_{1,2,\dots ,n-1}(\mathbb C^n)$. Since $T^n(\mathbb{C})$ acts smoothly and  properly  on $\mathrm{GL}_n(\mathbb C)$ by multiplication from the right, one has the identification
\begin{align*}
    \mathrm{GL}_n(\mathbb C) /T^n(\mathbb{C}) \cong F_{1,2,\dots ,n-1}(\mathbb C^n),
\end{align*}
and the smooth structure is the unique one turning the canonical projection into a smooth submersion. In particular, the complex dimension of $F_{1,2,\dots ,n-1}(\mathbb C^n)$ is given by: 
\begin{equation*}
n^2-\frac{n(n+1)}{2}=\frac{n(n-1)}{2}.    
\end{equation*}
In order to define a Riemannian metric on $F_{1,2,\dots ,n-1}(\mathbb C^n)$, we appeal to its compact realization:
\begin{align*}
    F_{1,2,\dots ,n-1}(\mathbb C^n)\cong \mathbf{U}(n)/ \mathbf{U}(1)^n,
\end{align*}
where $\mathbf U(n)$ is the \emph{unitary group}:
\begin{align*}
    \mathbf{U}(n) :=\{ M\in \mathbb{C}^{n\times n}\mid M^*M =I_{n}\} .
\end{align*}
This realization is actually obtained in a similar fashion as the one above and the action is given by by right multiplication from $\mathbf{U}(1)^n$, identified with the set of diagonal matrices in $\mathbf{U}(n)$. 
The Riemannian metric on $F_{1,2,\dots ,n-1}(\mathbb C^n)$ is then the unique one making the canonical projection $\pi: \mathbf{U}(n) \to \mathbf{U}(n)/ \mathbf{U}(1)^n$ a Riemannian submersion, where $\mathbf{U}(n)$ is equipped with its bi-invariant metric induced by the Killing form. This Riemannian submersion $\pi$ has totally geodesic fibers isometric to $\mathbf{U}(1)^n$, since it is a Bérard-Bergery fibration \cite[Theorem 9.80]{Besse2007-xr}.

One can also see $F_{1,2,\dots ,n-1}(\mathbb C^n)$ as an algebraic sub-variety of $\mathbb C P^{n-1} \times \dots \times \mathbb C P^{n-1}$, where $\mathbb C P^{n-1}$ is the complex projective space, i.e. the set of complex lines in $\mathbb C^n$. Indeed, one has the embedding of $F_{1,2,\dots ,n-1}(\mathbb C^n)$ into $\mathbb C P^{n-1} \times \dots \times \mathbb C P^{n-1}$ which is given by:
\[
(W_1,\dots,W_{n-1}) \to (W_1,W_2 \cap W_1^{\perp} ,\dots, W_{n-1} \cap W_{n-2}^\perp, W_{n-1}^\perp ) .
\]
This embedding is also a Riemannian immersion as can be seen from the following commutative diagram

\[
  \begin{tikzcd}
        (\mathbb{S}^{2n-1})^n \arrow[r, "\tilde{\pi}"] & (\mathbb{C}P^{n-1})^n \\
        \mathbf{U}(n) \arrow[r,"\pi" below] \arrow[u,"\iota"] & F_{1,2,\dots ,n-1}(\mathbb{C}^n) \arrow[u]
  \end{tikzcd} ,
\]
where $\iota$ is the Riemannian immersion of $\mathbf{U}(n)$ onto $(\mathbb{S}^{2n-1})^n$ column by column and $\tilde{\pi}$ is the Riemannian submersion which tensorizes the Hopf submersion $\mathbb{S}^{2n-1} \to \mathbb{C}P^{n-1}$.

We will parametrize (a dense subset of) $F_{1,2,\dots ,n-1}(\mathbb C^n)$ using  \emph{local affine} coordinates as follows. Let 
\[
\mathcal{D} :=\left\{ \left(\begin{matrix}
        a_{11} & \dots & a_{1n} \\
        \vdots & \ddots &\vdots \\
        a_{n1} & \dots & a_{nn}
    \end{matrix} \right) \in \mathbf{U}(n) \biggm| a_{n1} \neq 0, \dots, a_{nn} \neq 0   \right\}
\]
and consider the smooth map $p: \mathcal{D} \to \mathbb{C}^{n-1} \times \dots \times  \mathbb{C}^{n-1}$ defined by
\begin{align}\label{submersion p}
p\left(\begin{matrix}
        a_{11} & \dots & a_{1n} \\
        \vdots & \ddots &\vdots \\
        a_{n1} & \dots & a_{nn}
    \end{matrix} \right) 
    =\left( \left(\begin{matrix}
        a_{11}/a_{n1} \\
        \vdots \\
        a_{(n-1)1}/a_{n1}
        \end{matrix}
        \right), \dots, \left(\begin{matrix}
        a_{1n}/a_{nn} \\
        \vdots \\
        a_{(n-1)n}/a_{nn}
        \end{matrix}
        \right) \right).
\end{align}
It is not difficult to see that for every $M_1,M_2 \in \mathcal{D}$, $p(M_1)=p(M_2)$ is equivalent to $M_2=M_1g$ for some $g \in \mathbf{U}(1)^n$ (any two such matrices differ by a diagonal unitary matrix). Since $p$ is a submersion from $\mathcal{D}$ onto its image $\mathcal{O}:=p(\mathcal{D})$, one deduces that there exists a diffeomorphism $\Psi$ between an open dense subset of $\mathbf{U}(n)/ \mathbf{U}(1)^n$ and $p(\mathcal{D})$ such that $\Psi \circ \pi =p$. This gives rise to the local set of coordinates  on $F_{1,2,\dots ,n-1}(\mathbb C^n)$. Those coordinates are compatible with the metric in the sense that $p$ is a Riemannian submersion, which implies that $\Psi$ is an isometry. Notice that $\mathcal{O}$ can explicitly be described as
\[
\mathcal{O}=\left\{ w=(w_1,\dots,w_{n-1},w_n) \in \mathbb{C}^{n-1} \times \dots \times \mathbb{C}^{n-1}  \mid w_i^* w_j =-1,\,  1 \le i <j \le n \right\},
\]
which yields a nice parametrization of a dense open subset of $F_{1,2,\dots ,n-1}(\mathbb C^n)$ by the algebraic manifold $\mathcal{O}$. This parametrization will be extensively used in the sequel. 

\subsection{Jacobi polynomials on simplices}\label{section Jacobi}
This section introduces the Jacobi operator in the simplex and describes its spectral resolution. The latter consists of a discrete spectrum and a family of eigenfunctions given by a multivariable extension of Jacobi polynomials, referred to as Jacobi polynomials in the simplex. These polynomials were constructed by T. Koornwinder in the two-variable setting and his construction reflects the right-neutrality of the Dirichlet distribution. We refer the reader to the monograph \cite{MR3289583}, 
and the papers \cite{aktacs2013sobolev} and \cite{MR2817619} for further details. 

Consider the simplex 
\begin{equation*}
 \Sigma_{n-1} :=    \{ \lambda \in \mathbb{R}^{n-1} \mid  \lambda_j \geq 0, \, 1 \leq j \leq n-1, \quad \lambda_1 + \dots + 
    \lambda_{n-1} \leq 1\}. 
\end{equation*}
The Jacobi operator in $\Sigma_{n-1}$ is then defined by
\begin{equation}\label{JacSim}
\mathcal{G}_\kappa :=\sum_{j=1}^{n-1} \lambda_j(1-\lambda_j) \frac{\partial^2}{\partial \lambda_{j}^2} 
 + \sum_{j=1}^{n-1}\left[\left(\kappa_j+\frac{1}{2}\right) 
 - \left(|\kappa| +\frac{n}{2}\right)\lambda_j\right] \frac{\partial}{\partial \lambda_{j}} 
 -\sum_{1 \leq j \neq \ell \leq n-1}\lambda_j\lambda_{\ell}  \frac{\partial^2}{\partial \lambda_{j}\lambda_{\ell}} .   
 \end{equation}
 Here, $\kappa = (\kappa_1, \dots, \kappa_n)$ is a parameter set such that $\kappa_j > -1/2$ for any $1 \leq j \leq n$ and  $|\kappa| = \kappa_1 + \dots + \kappa_n$. The operator $\mathcal{G}_\kappa$ is symmetric with respect to the Dirichlet measure on $\Sigma_{n-1}$ whose density is given by: 
\begin{multline}\label{DenDirich}
W^{(\kappa)}(\lambda_1, \dots, \lambda_{n-1}) := \frac{\Gamma(|\kappa|+(n/2))}{\prod_{j=1}^n\Gamma(\kappa_j+(1/2))}  \lambda_1^{\kappa_1-(1/2)}\cdots
\lambda_{n-1}^{\kappa_{n-1}-(1/2)} 
\\ (1-\lambda_1-\cdots-\lambda_{n-1})^{\kappa_n-(1/2)}. 
\end{multline}
The spectrum of $\mathcal{G}_\kappa$ is discrete and is given by 
\begin{equation}\label{eq-eigenv-prelim}
-j\left( j+|\kappa| +\frac{n-2}{2}\right), \quad j \geq 0.
\end{equation}
The corresponding set of orthonormal eigenfunctions consists of the so-called Jacobi polynomials in the simplex. Given a multi-index $\tau \in \mathbb{N}^{n-1}$ 
with total weight
\begin{equation*}
|\tau| := \tau_1 + \dots + \tau_{n-1},  
\end{equation*}
the corresponding Jacobi polynomial has degree $|\tau|$ and admits the following explicit formula:
\begin{align*}
    P_{\tau}^{(\kappa)}(\lambda_1,\dots ,\lambda_{n-1}):=\frac{1}{\sqrt{C_{\tau}(\kappa)}}
    \prod_{j=1}^{n-1}\left(1-\sum_{i=1}^{j-1}\lambda_i\right)^{\tau_j} 
    P^{(a_j,\kappa_i-(1/2))}_{\tau_j}\left(\frac{2\lambda_j}{1-\sum_{i=1}^{j-1}\lambda_i} -1\right) ,
\end{align*}
where $P^{(\alpha ,\beta)}_m$ stands for the $m^{\textrm{th}}$ Jacobi polynomial of index $(\alpha ,\beta)$, 
$$a_j :=2\sum_{i=j+1}^{n-1}\tau_i +2\sum_{i=j+1}^{n-1}\kappa_i +\frac{1}{2}(n-j-2),$$ 
and 
\begin{align*}
    C_{\tau}(\kappa) := \frac{1}{(|\kappa| + (n/2))_{2|\tau|}}
    \prod_{j=1}^{n-1}\frac{(a_j+\kappa_j+(1/2))_{2\tau_j}(a_j+1)_{\tau_j}(\kappa_j+(1/2))_{\tau_j}}{(a_j+\kappa_j+(1/2))_{\tau_j}\tau_j!} .
 \end{align*}

As such, the density with respect to the Dirichlet measure $W^{(\kappa)}$ of the heat semi-group $e^{t\mathcal{G}_\kappa}$ reads 

\begin{equation}\label{kernel jacobi simplex}
q^{(\kappa_1,\dots ,\kappa_{n-1},\kappa_n)}_t(x,y) =\sum_{\tau\in\mathbb{N}^{n-1}} 
e^{-|\tau |(|\tau |+|\kappa |+(n-2)/2)t} P_{\tau}^{(\kappa)}(x)P_{\tau}^{(\kappa)}(y).  
\end{equation}
 A diffusion with generator $\mathcal{G}_\kappa$ is called a Jacobi diffusion in the simplex $\Sigma_{n-1}$. 

For symmetry reasons, it will sometimes be useful to lift Jacobi diffusions in $\Sigma_{n-1}$ to diffusions in the $n-1$ simplex of $\mathbb{R}^n$.  More precisely, define
\begin{align}\label{eq-simplex}
    \mathcal{T}_n :=\{\lambda\in\mathbb{R}^{n}\mid\lambda_j\geq 0, \, 1\leq j\leq n,\, \lambda_1 +\dots +\lambda_n =1\}.
\end{align}
It is easy to check that if $(\lambda_1(t),\dots,\lambda_{n-1}(t))$ is a diffusion with generator $\mathcal{G}_\kappa$, then $(\lambda_1(t),\dots,\lambda_{n}(t))$, where $\lambda_n(t)=1-\sum_{k=1}^{n-1} \lambda_k(t)$, is a diffusion in $\mathcal{T}_n$ with generator
\begin{equation}\label{GenJacSim1}
\widehat{\mathcal{G}}_\kappa:=\sum_{j=1}^{n} \lambda_j(1-\lambda_j) \frac{\partial^2}{\partial \lambda_{j}^2} 
 + \sum_{j=1}^{n}\left[\left(\kappa_j+\frac{1}{2}\right) 
 - \left(|\kappa| +\frac{n}{2}\right)\lambda_j\right] \frac{\partial}{\partial \lambda_{j}} 
 -\sum_{1 \leq j \neq \ell \leq n}\lambda_j\lambda_{\ell}  \frac{\partial^2}{\partial \lambda_{j}\partial \lambda_{\ell}} . 
 \end{equation}

The operator $\widehat{\mathcal{G}}_\kappa$ will be called the lift of $\mathcal{G}_\kappa$ to $\mathcal{T}_n$ and a diffusion with generator $\widehat{\mathcal{G}}_\kappa$ will be referred to as a Jacobi diffusion in $ \mathcal{T}_n$.

\section{Brownian motion on the full flag manifold}\label{sec:BM-on-FF}
In this section, we study Brownian motion processes on the full flag manifold $F_{1,2,\dots ,n-1}(\mathbb{C}^n)$. We begin by recalling the unitary Brownian motion, then introduce Brownian motion on the full flag manifold using the parametrization developed in Section \ref{sec-full-flag}. Finally, we study the associated radial processes and establish connections to Jacobi processes on the simplex.

\subsection{Unitary Brownian motion}
In this paragraph, we recall the definition of the (left) Brownian motion on the unitary group $\mathbf{U}(n)$. For further details we refer to \cite[Section 3.5]{book}. The unitary group
\begin{align*}
    \mathbf{U}(n) :=\left\{ M \in \mathbb C^{n\times n}, M^*M={I}_n \right\}
\end{align*}\index{unitary group}
is a compact simple subgroup of the general linear group and its Lie algebra
\begin{align*}
    \mathfrak{u}(n)=\left\{ A \in \mathbb C^{n\times n}, A^*+A=0 \right\}
\end{align*}
is the vector space of skew-Hermitian matrices.
One can equip $\mathfrak{u}(n)$ with the inner product:
\[
B(A_1,A_2)=-\frac{1}{2}\textrm{tr}(A_1A_2), \quad A_1, A_2\in \mathfrak{u}(n),
\]
known as the Killing form, which induces on $\mathbf{U}(n)$ a bi-invariant Riemannian metric.

Note that with respect to the Killing form an orthonormal basis of $\mathfrak{u}(n)$ can be given by
\[
\{E_{\ell j}-E_{j \ell}, i(E_{\ell j}+E_{j\ell}), T_\ell\mid 1\le \ell <j \le n\} ,
\]
where $E_{\ell j} =(\delta_{(\ell, j)}(k,m))_{1 \le k , m \le n}$, and $T_\ell=\sqrt{2}iE_{\ell\ell}$. Therefore the Brownian motion $(A(t))_{t \ge 0}$ on $\mathfrak{u}(n)$ is  of the form
\begin{align}\label{eq-At}
A(t) =\sum_{1\le \ell<j\le n}(E_{\ell j}-E_{j\ell})B_{\ell j}(t)+i\sum_{1\le \ell<j\le n}(E_{\ell j}+E_{j\ell})\tilde{B}_{\ell j}(t)+\sum_{j=1}^n T_{j}{B}_j(t),\quad t\ge0, 
\end{align}
where $B_{\ell j}$, $\tilde{B}_{\ell j}$, ${B}_j$ are  independent standard real Brownian motions. Denote
\begin{equation}\label{eq-At-ij}
A(t) =\sum_{1\le \ell,j\le n}A_{\ell j}(t)E_{\ell j},\quad t\ge 0.
\end{equation}
Then, for any $1\le \ell \not= j\le n$ the quadratic variations of its off diagonal entries are given by
\begin{equation}
\label{Brack1} 
dA_{\ell j}(t)\, d\overline{A}_{\ell j}(t)=2dt,\quad dA_{\ell j}(t)\, d{A}_{\ell j}(t) =0 ,
\end{equation}
while for any $1\le \ell\le n$, we have:
\begin{equation}\label{Brack2}
dA_{\ell \ell}(t)\, d\overline{A}_{\ell \ell}(t)=2dt, \quad dA_{\ell \ell}(t)\, d{A}_{\ell \ell}(t)=-2dt.
\end{equation}

The Brownian motion on $\mathbf{U}(n)$ thus satisfies the stochastic differential equation in Stratonovitch form:

\begin{equation}\label{eq-UBM-SDE}
dU(t)=U(t)\circ dA(t).
\end{equation}

In It\^o's form, this stochastic differential equation reads:
\begin{align}\label{eq-sde-U}
  dU(t)=U(t) dA(t)-nU(t) dt.  
\end{align}

\subsection{Generator of the Brownian motion on the full flag manifold}

In this section we compute the generator of the Brownian motion on the full flag manifold $F_{1,2,\dots ,n-1}(\mathbb{C}^n) \simeq 
\mathbf{U}(n)/\mathbf{U}(1)^n$. We use the parametrization explained in the Section \ref{sec:preliminaries}.  
Let $U(t)=(U_{ij}(t))_{1\le i,j\le n}$, $t\ge0,$ be a Brownian motion on $\mathbf{U}(n)$, which is started from a point $U(0)  \in \mathcal{D}$ as before. Since the map $p:\mathcal{D} \to \mathcal{O}$ defined by \eqref{submersion p} is a Riemannian submersion with totally geodesic fibers isometric to $\mathbf{U}(1)^n$ and since
\[
\mathbb{P} ( \exists \, t \ge 0,  U(t) \notin \mathcal{D}) =0,
\]
 one deduces that the process $w(t)=(w_{jk}(t))_{1\le j,k\le n-1}$ defined by 
 
\begin{align}\label{eq-w-def}
    w_{kj}(t) :=\frac{U_{kj}(t)}{U_{nj}(t)},\quad 1\le k \le n-1, \quad 1\le j \le n,
\end{align}
parametrizes a Brownian motion on $F_{1,2,\dots ,n-1}(\mathbb{C}^n)$. We denote the $j$-th column of $w(t)$ by 
\[
w_j(t):=(w_{1j}(t),\dots, w_{(n-1)j}(t))^{T},\quad 1\le j\le n,
\]
and set
\begin{align}\label{eq-rj}
 r_j(t):=\sqrt{|w_{1j}(t)|^2+\cdots +|w_{(n-1)j}(t)|^2},
 \quad 1 \leq j \leq n.
\end{align}
In this respect, the orthogonality of $U$ implies that
\begin{equation}\label{OR1}
\frac1{|U_{nj}|^2}=1+r_j^2, \quad 1\le j\le n,
\end{equation}
and we also recall that for all $1\le j\not=\ell\le n$,
\begin{align}\label{eq-orth-w}
\sum_{s=1}^{n-1} w_{sj}\cj{w}_{s\ell}=-1.
\end{align}

\begin{lemma}\label{lemma:expression_wt}
Let $(w(t))_{t\geq 0}$ be the Brownian motion on the full flag manifold $F_{1,2,\dots ,n-1}(\mathbb{C}^n)$ as given in \eqref{eq-w-def}. It satisfies the stochastic differential equation
\begin{align}\label{eq-w-SDE}
    dw_{kj}=\sum_{s=1}^n\frac{U_{ks}-w_{kj}U_{ns}}{U_{nj}}dA_{sj},\quad 1\le k\le n-1,1\le j\le n,
\end{align}
where $(A(t))_{t \ge 0}$  is a Brownian motion on $\mathfrak{u}(n)$ as given in \eqref{eq-At}.
\end{lemma}
\begin{proof}
From \eqref{eq-w-def}  we can deduce the following stochastic differential equation in It\^o's sense:
    \begin{align}\label{eq-dw-1}
        dw_{kj}=U_{nj}^{-1}dU_{kj}-\frac{U_{kj}}{U_{nj}^2}dU_{nj}+dU_{kj}dU_{nj}^{-1} +U_{kj}\frac{dU_{nj}dU_{nj}}{U_{nj}^3}.
    \end{align}
By rewriting \eqref{eq-sde-U} coordinate-wise, we obtain that
\[
dU_{kj}=\sum_{s=1}^{n}U_{ks}dA_{sj}-n U_{kj}dt,\quad \mbox{for all }k=1,\dots,n .
\]
Plugging it into \eqref{eq-dw-1} and appealing to \eqref{Brack1} and to \eqref{Brack2}, we then obtain for any $1\le k\le n-1$ and any $1\le j\le n-1$ that
\begin{align*}
    dw_{kj}
    &=\bigg(\sum_{s=1}^n\frac{U_{ks}}{U_{nj}}dA_{sj}-nw_{kj}dt\bigg)-\bigg(\sum_{s=1}^n\frac{w_{kj}U_{ns}}{U_{nj}}dA_{sj} -nw_{kj}dt\bigg)+2w_{kj}dt-2w_{kj}dt\\
    &=\sum_{s=1}^n\frac{U_{ks}-w_{kj}U_{ns}}{U_{nj}}dA_{sj}.
\end{align*}
\end{proof}

Using Lemma \ref{lemma:expression_wt}, we shall compute the quadratic variation of $w(t)$. 

\begin{lemma}\label{lemma-qv-dw}
Consider the  Brownian motion $w(t)= (w_{kj}(t))_{1\le k\le n-1,1\le j \le n}$ on the full flag manifold $F_{1,2,\dots ,n-1}(\mathbb{C}^n)$ as in Lemma \ref{lemma:expression_wt}.
Then the quadratic variations of its entries are given by:
    \begin{align}
       &dw_{kj}d\cj{w}_{m\ell}=\delta_{j\ell} (1+r_j^2)\left( \delta_{km}+w_{kj}\cj{w}_{mj} \right)(2dt), \label{eq-dw-quad-1}\\     
       &dw_{kj}d{w}_{m\ell}=-(w_{k\ell}-w_{kj})(w_{mj}-w_{m\ell})(2dt) , \label{eq-dw-quad-2}\\
        &d\cj{w}_{kj}d{\cw}_{m\ell}=-(\cw_{k\ell}-\cw_{kj})(\cw_{mj}-\cw_{m\ell})(2dt)  \label{eq-dw-quad-3}
    \end{align}
   for any $1\le k,m\le n-1$, $1\le j,\ell\le n$, where $r_j$ is defined as in \eqref{eq-rj}.
\end{lemma}
\begin{proof}
    Using \eqref{eq-w-SDE} we have 
    \begin{align}\label{eq-dw-mid1}
    dw_{kj}d\cj{w}_{m\ell}=\sum_{s,p=1}^n\frac{U_{ks}-w_{kj}U_{ns}}{U_{nj}}\frac{\cj{U}_{mp}-\cj{w}_{m\ell}\cj{U}_{np}}{\cj{U}_{n\ell}}dA_{sj}d\cj{A}_{p\ell}.
\end{align}
Using \eqref{eq-At} we can easily compute that 
\[
dA_{sj}d\cj{A}_{p\ell}=\delta_{sp}\delta_{j\ell}\,(2dt).
\]
Plugging it into \eqref{eq-dw-mid1} we then obtain 
    \begin{align*}
    dw_{kj}d\cj{w}_{m\ell}
    &=\delta_{j\ell}(2dt) \sum_{s=1}^n\frac{U_{ks}-w_{kj}U_{ns}}{U_{nj}}\frac{\cj{U}_{ms}-\cj{w}_{mj}\cj{U}_{ns}}{\cj{U}_{nj}}\\
    &=\frac{\delta_{j\ell} }{|U_{nj}|^2}\left( \delta_{km}+w_{kj}\cj{w}_{mj} \right)(2dt),
\end{align*}
which yields \eqref{eq-dw-quad-1}, after using the relation \eqref{OR1}. Similarly for \eqref{eq-dw-quad-2}, we compute 
\begin{align*}
    dw_{kj}d{w}_{m\ell}
    &=\sum_{s,p=1}^n\frac{U_{ks}-w_{kj}U_{ns}}{U_{nj}}\frac{{U}_{mp}-{w}_{m\ell}{U}_{np}}{{U}_{n\ell}}dA_{sj}d{A}_{p\ell}\\
    &=-(2dt) \frac{U_{k\ell}-w_{kj}U_{n\ell}}{U_{nj}}\frac{{U}_{mj}-{w}_{m\ell}{U}_{nj}}{{U}_{n\ell}}=-(2dt)(w_{k\ell}-w_{kj})(w_{mj}-w_{m\ell}),
\end{align*}
where the second equality follows from
$dA_{sj}d{A}_{p\ell}=-dA_{sj}d\overline{A}_{\ell p} =-\delta_{s\ell}\delta_{jp}\,(2dt)$.
Finally, the identity \eqref{eq-dw-quad-3} follows readily by taking the complex conjugate of \eqref{eq-dw-quad-2}.
\end{proof}

We are now ready to compute the generator of $w(t)$.
\begin{prop}\label{Laplacian full flag}
    The stochastic process $w(t)$, $t\ge0,$ given in \eqref{eq-w-def} is a diffusion process with generator $\frac{1}{2} \Delta$ where $\Delta$  is given by
     the following second order differential operator on smooth functions on $\mathbb{C}^{n-1} \times \dots \times \mathbb{C}^{n-1}$:
    \begin{align}\label{eq-Del-n-1}
        \Delta=&4\sum_{j=1}^{n}(1+r_j^2)\sum_{k,m=1}^{n-1}(\delta_{km}+w_{kj}\cj{w}_{mj}) \frac{\partial^2 }{\partial w_{kj}\partial \cw_{mj}} \\
         &-2\sum_{1\leq j\neq \ell\leq n} \sum_{k,m=1}^{n-1} (w_{k\ell}-w_{kj})(w_{mj}-w_{m\ell}) \frac{\partial^2 }{\partial w_{kj}\partial w_{m\ell}} \notag \\
         & -2\sum_{1\leq j\neq\ell\leq n} \sum_{k,m=1}^{n-1} (\cw_{k\ell}-\cw_{kj})(\cw_{mj}-\cw_{m\ell}) \frac{\partial^2 }{\partial \cw_{kj}\partial \cw_{m\ell}},
\end{align}
where we recall $r_j=\sqrt{|w_{1j}|^2+\cdots +|w_{(n-1)j}|^2}$.
\end{prop}
\begin{proof}
   Recall It\^o's formula for complex semimartingales: for any $Z(t)=(Z_1(t),\dots, Z_N(t))$, $t\ge0$, and any complex function $f$, we have
    \begin{align*}
   &d [f (Z(t) )]  =\sum_{j=1}^N\left( \frac{\partial f }{\partial z_j} (Z(t)) dZ_j(t) + \frac{\partial f }{\partial \overline{z}_j} (Z(t)) d\overline{Z}_j(t)\right) \\
    &\quad+\frac{1}{2} \sum_{j,\ell=1}^N  \left(\frac{\partial^2 f }{\partial z_j\partial z_\ell} (Z(t))dZ_i(t)dZ_j(t)+2\frac{\partial^2 f }{\partial z_j\partial \overline{z}_\ell} (Z(t))dZ_j(t)d\overline{Z}_\ell(t)+ \frac{\partial^2 f }{\partial \overline{z}_j\partial \overline{z}_\ell} (Z(t))d\overline{Z}_j(t)d\overline{Z}_\ell(t) \right).
    \end{align*}
   Applying this formula to the process $w(t)= (w_{jk}(t))_{1\le j,k\le n-1}$ and using Lemma \ref{lemma-qv-dw}, \eqref{eq-Del-n-1} follows after straightforward computations. 
\end{proof}
\begin{remark}
    Note that the restriction of $\Delta/2$ to functions depending only on the $j$-th column yields the infinitesimal generator of the $j$-th column diffusion $w_j(t):=(w_{1j}(t),\dots, w_{(n-1)j}(t))^{T}$, $t\ge0$: 
    \[
    2(1+r_j^2)\sum_{k,m=1}^{n-1}(\delta_{km}+w_{kj}\cj{w}_{mj}) \frac{\partial^2 }{\partial w_{kj}\partial \cw_{mj}}.
    \]
    This is indeed the generator in local affine coordinates of a Brownian motion on the complex projective space 
    \begin{equation*} 
    \mathbb{CP}^n:=\frac{\mathbf{U}(n)}{\mathbf{U}(1)\mathbf{U}(n-1)},
    \end{equation*}
    see \cite[Section 5.1]{book}. Therefore the columns of $w(t)$ are Brownian motions on $\mathbb{CP}^n$. However, they are of course not independent as can be seen from the cross terms  $\frac{\partial^2 }{\partial w_{kj}\partial w_{m\ell}}$ in formula \eqref{eq-Del-n-1}. 
    
\end{remark}

\subsection{Radial motions}
The main object of study in this section is the Jacobi process that is associated with the radial processes. Recall the Brownian motion $(w(t))_{t\geq 0}$ on $F_{1,2,\dots ,n-1}(\mathbb{C}^n)$ and its $j$-th column radial process $r_j(t)$, $t\ge0$, as defined in \eqref{eq-rj}. Consider the process
\begin{align}\label{eq-lambda}
\lambda(t):=\left(\frac{1}{1+r_1(t)^2},\dots,\frac{1}{1+r_{n}(t)^2} \right), \quad t \ge 0.
\end{align}
In the theorem below we compute the SDE satisfied by $\lambda(t)$ and its generator.
\begin{theorem}\label{Th1}
Let $\mathcal{T}_n$ be the simplex as in \eqref{eq-simplex}.
The process $(\lambda (t))_{t\geq 0}$, given as in \eqref{eq-lambda}, satisfies the stochastic differential equation
    \begin{align*}
        d\lambda_j =2(1-n\lambda_j) dt+2 \sum_{\ell=1, \ell \neq j}^n  \sqrt{ \lambda_{\ell }\lambda_j} d\gamma_{\ell j}, \quad 1 \le j \le n,
    \end{align*}
    where $(\gamma_{\ell j}(t))_{\ell < j}$, $t\geq 0$, is a Brownian motion on $\mathbb{R}^{\frac{1}{2}n(n-1)}$ and $\gamma_{\ell j} :=-\gamma_{j\ell}$ for $\ell > j$.
Consequently, $(\lambda (t))_{t\geq 0}$ is a Jacobi process  in the simplex $\mathcal{T}_n$ with generator $2\widehat{\mathcal{G}}_{1/2,\dots,1/2}$, where 

\begin{align}\label{Gen2}
\widehat{\mathcal{G}}_{1/2,\dots,1/2}  &=  
 \sum_{j=1}^{n} \lambda_j(1-\lambda_j) \frac{\partial^2}{\partial \lambda_{j}^2} 
 + \sum_{j=1}^{n}\left(1 - n\lambda_j\right) \frac{\partial}{\partial \lambda_{j}} 
 -\sum_{1 \leq j \neq \ell \leq n}\lambda_j\lambda_{\ell}  \frac{\partial^2}{\partial \lambda_{j}\partial\lambda_{\ell}}
\end{align}
is the Jacobi operator with parameter $\kappa=(1/2,\dots,1/2)$.

\end{theorem}
\begin{proof}
    Applying Itô's formula to \eqref{eq-lambda} we obtain
    \begin{align}\label{Formula w: proof lambda}
        d\lambda_j =d\left(\frac{1}{1+w_j^*w_j}\right) 
        =-\frac{dw_j^*w_j + w_j^*dw_j +dw_j^*dw_j}{(1+w_j^* w_j )^2}+\frac{(dw_j^*w_j + w_j^*dw_j)^2 }{(1+w_j^* w_j )^3} .
    \end{align}

    By Lemma \ref{lemma:expression_wt}  we have

    \begin{align*}
    dw_j^*w_j + w_j^*dw_j =\sum_{\ell=1, \ell \neq j}^n \frac{\sum_{\alpha =1}^n\overline{U}_{\alpha\ell}w_{\alpha j} -r_j^2\overline{U}_{n\ell}}{\overline{U}_{nj}}d\overline{A}_{\ell j} + \sum_{\ell=1, \ell \neq j}^n\frac{\sum_{\alpha =1}^n\overline{w}_{\alpha j}U_{\alpha\ell} -r_j^2U_{n\ell }}{U_{nj}}dA_{\ell j}.
    \end{align*}
    
    From \eqref{OR1} and \eqref{eq-lambda} we know that $|U_{nj}|^2 =\lambda_j$, therefore $U_{nj} =e^{i\varphi_j}\sqrt{\lambda_j}$ for some real valued random variable $\varphi_j$.
    Using this, together with the relation $w_{kj} =U_{kj}U^{-1}_{nj}$ and equation \eqref{eq-orth-w}, we obtain

    \begin{align*}
    dw_j^*w_j + w_j^*dw_j       &=\sum_{\ell=1, \ell \neq j}^n \left(\sum_{\alpha =1}^n\overline{w}_{\alpha\ell}w_{\alpha j} -r_j^2\right)\frac{\overline{U}_{n\ell}}{\overline{U}_{nj}}d\overline{A}_{\ell j} + \sum_{\ell=1, \ell \neq j}^n \left(\sum_{\alpha =1}^n\overline{w}_{\alpha j}w_{\alpha\ell} -r_j^2\right)\frac{U_{n\ell }}{U_{nj}}dA_{\ell j} \\
     &= - \sum_{\ell=1, \ell \neq j}^n \left( 1 +r_j^2\right)\sqrt{\frac{\lambda_{\ell }}{\lambda_{j}}}e^{-i(\varphi_{\ell} -\varphi_j)}d\overline{A}_{\ell j}  -\sum_{\ell=1, \ell \neq j}^n\left( 1 +r_j^2\right)\sqrt{\frac{\lambda_{\ell}}{\lambda_{j}}}e^{i(\varphi_{\ell} -\varphi_j)}dA_{\ell j}.
    \end{align*}

     Using \eqref{eq-At} and \eqref{eq-At-ij} one can easily verify that
     \begin{equation*}
     d\gamma_{\ell j}:=\frac{1}{2}(e^{i(\varphi_{\ell} -\varphi_j)}dA_{\ell j} +e^{-i(\varphi_{\ell } -\varphi_j)}d\overline{A}_{\ell j})=\cos (\varphi_{\ell j})dB_{\ell j}^1-\sin (\varphi_{\ell j})dB_{\ell j}^2
     \end{equation*}
     and $\gamma_{\ell j}=-\gamma_{j\ell}$, where $B_{\ell j} =B_{lj}^1 +iB_{lj}^2$ are the complex valued Brownian motions from \eqref{eq-At}. By the Lévy characterisation theorem, $(\gamma_{\ell j} (t))_{\ell <j}$, $t\geq 0$, is standard Brownian motion on $\mathbb{R}^{\frac{1}{2}n(n-1)}$. Therefore, we obtain
     \begin{align}\label{eq-dwdw*-1}
          dw_j^*w_j + w_j^*dw_j  =- 2 \sum_{\ell=1, \ell \neq j}^n  \sqrt{\frac{\lambda_{\ell }}{\lambda_{j}^3}} d\gamma_{\ell j}
     \end{align}
     for $1\leq j\leq n$.
     Consequently,
     \begin{align}\label{eq-dwdw*-2}
     (dw_j^*w_j + w_j^*dw_j)^2  =4 \sum_{\ell=1, \ell \neq j}^n  \frac{\lambda_{\ell }}{\lambda_{j}^3} dt=4  \frac{1-\lambda_j}{\lambda_{j}^3} dt .
      \end{align}
     On the other hand, from Lemma \ref{lemma-qv-dw} we have
     \begin{align}\label{eq-dwdw*-3}
     dw_j^*dw_j=\sum_{k=1}^{n-1} d\bar{w}_{kj}dw_{kj}=2(1+r_j^2)(n-1+r_j^2)dt=\frac{2}{\lambda_j} \left(n-2+\frac{1}{\lambda_j}\right)dt.
      \end{align}
    Plugging \eqref{eq-dwdw*-1} and \eqref{eq-dwdw*-3} into \eqref{Formula w: proof lambda} we end up with:
       \begin{align*}
        d\lambda_j 
         &=2 \sum_{\ell=1, \ell \neq j}^n  \sqrt{ \lambda_{\ell }\lambda_j} d\gamma_{\ell j}-2 \left((n-2)\lambda_j+1 \right)dt+4  (1-\lambda_j)dt \\
         &=2(1-n\lambda_j) dt+2 \sum_{\ell=1, \ell \neq j}^n  \sqrt{ \lambda_{\ell }\lambda_j} d\gamma_{\ell j}, \quad 1\leq  j\leq n.
    \end{align*}
  From its definition we know that $\lambda_j=\frac{1}{1+r_j^2}=|U_{nj}|^2$. Hence 
  \[
  \sum_{j=1}^n\lambda_j=U_n^*U_n=1,
  \]
  which implies that the process $(\lambda(t))_{t\ge0}$ lives in the simplex $\mathcal{T}_n$. Lastly, one obtain the generator $2\widehat{\mathcal{G}}_{1/2,\dots,1/2}$ of $\lambda(t)$ following standard computations.
     \end{proof}

     \begin{remark}
         Since 
         \begin{equation*}
          1+r_j^2 = \frac{1}{|U_{nj}|^2}, \quad 1\leq j\leq n,   
         \end{equation*}
one has $\lambda(t) = (|U_{n1}(t)|^2, \dots, |U_{nn}(t)|^2)$.  
Since the last row vector of the random matrix $U(t)$ is a Brownian motion on the sphere, then the joint distribution of any $k$-tuple 
\begin{equation*}
(|U_{n1}(t)|^2, \dots, |U_{nk}(t)|^2), \quad 1\leq k \leq n,   \end{equation*}
was determined in \cite{MR3322612} using the decomposition of unitary spherical harmonics under the action of the unitary group $\mathbf{U}(k)$. Moreover, the corresponding  infinitesimal generator was informally determined using integration by parts and coincides when $k=n$ with 
$\widehat{\mathcal{G}}_{1/2,\dots,1/2}$. As a matter of fact, Theorem \ref{Th1} provides a direct derivation of this generator.
\end{remark}

\section{Stochastic area functionals and skew-product decompositions}\label{sec:SAF-and-skew-product}
In this section, we introduce the stochastic area functionals associated with Brownian motion on the full flag manifold. We then derive explicit expressions for their characteristic functions and prove that these functionals converge in distribution to a multivariate Cauchy distribution.
\subsection{The unitary group as a torus bundle} 

Recall that one can see the complex full flag manifold $F_{1,2,\dots ,n-1}(\mathbb C^n)$ as the Riemannian homogeneous space $\mathbf U(n)/\mathbf{U}(1)^n$, where the action is the one by right multiplication from $\mathbf{U}(1)^n$, identified with the set of diagonal matrices in 
$\mathbf{U}(n)$. This yields a fibration
\begin{align}\label{torus bundle}
\mathbf{U}(1)^n \to \mathbf U(n) \to F_{1,2,\dots ,n-1}(\mathbb C^n) ,
\end{align}
which allows us to see the unitary group $\mathbf U(n)$ as a torus $\mathbf{U}(1)^n$-bundle over $F_{1,2,\dots ,n-1}(\mathbb C^n)$. As already pointed out, the canonical projection $\pi: \mathbf{U}(n) \to \mathbf{U}(n)/ \mathbf{U}(1)^n$ is a Riemannian submersion with totally geodesic fibers isometric to $\mathbf{U}(1)^n$. The horizontal space of that submersion is denoted by $\mathcal{H}$, i.e. $\mathcal{H}$ is the orthogonal complement of the kernel of the derivative of $\pi$. The vertical space of the submersion, i.e. the kernel of the derivative of $\pi$, will be denoted by $\mathcal{V}$. We will consider the following vector fields on $\mathbf U(n)$ given at $M \in \mathbf U(n)$ by
\[
\frac{\partial f }{\partial \theta_j} (M) := \left.\frac{d}{ds} \right|_{ s=0} f \left( Me^{is E_{jj}} \right) ,
\]
where we use the same notation as before $E_{j j} =(\delta_{(j, j)}(k,m))_{1 \le k , m \le n}$. This notation is consistent with the fact that if we simply parametrize 
$\mathbf{U}(n)$ as a subset of the set of matrices $$\left\{ \left(\begin{matrix}
        a_{11} & \dots & a_{1n} \\
        \vdots & \ddots &\vdots \\
        a_{n1} & \dots & a_{nn}
    \end{matrix} \right) \Bigm| a_{ij} \in \mathbb{C} \right\},
$$
then it is plain that:
\[
\frac{\partial }{\partial \theta_j} =i \sum_{k=1}^n \left(  a_{kj} \frac{\partial}{\partial a_{kj}}-\overline{a}_{kj} \frac{\partial}{\partial \overline{a}_{kj}} \right).
\]
Notice that the vector fields $\frac{\partial  }{\partial \theta_j}$, $1\leq j\leq n $, commute and form at any point a basis of the vertical space $\mathcal{V}$.

\begin{lemma}
Consider the $\mathbb{R}^n$-valued one-form on $\mathbf{U}(n)$ given by
\[
\eta=(\eta_1,\dots,\eta_n) ,
\]
where
\begin{align}\label{connection form 3}
\eta_j :=\frac{1}{2i} \sum_{k=1}^{n}\left(  \overline{a}_{kj} da_{kj} -a_{kj} d\overline{a}_{kj}   \right).
\end{align}
Then, $\eta$ is the connection form of the torus bundle \eqref{torus bundle}, that is:

\begin{itemize}
\item[(i)] for every $g \in \mathbf{U}(1)^n$, $g^*\eta=\eta$ (invariance of $\eta$ with respect to the group action);
\item[(ii)] $\eta_j \left( \frac{\partial}{\partial \theta_i} \right)=\delta_{ij}$;
\item[(iii)] $\mathrm{ker} (\eta)=\mathcal{H}$.
\end{itemize}

\end{lemma}
 
\begin{proof}
 
 The one-form $\eta_j$ is the contact form of the unit sphere
 \[
 \mathbb{S}^{2n-1} :=\left\{ (a_{1j},\dots,a_{nj}) \in \mathbb{C}^n \biggm| \sum_{k=1}^n |a_{kj}|^2 =1  \right\}
 \]
 and its kernel is the horizontal space of the Hopf submersion $\mathbb{S}^{2n-1} \to \mathbb{C}P^{n-1}$. Therefore $\eta_j$ is $\mathbf{U}(1)$-invariant and satisfies $\eta_j \left( \frac{\partial}{\partial \theta_j} \right)=1$. The properties (i), (ii) and (iii) then easily follow.
 \end{proof}
 
 We will consider a convenient local trivialization of $\mathbf{U}(n)$ seen as a torus $\mathbf{U}(1)^n$-bundle over $F_{1,2,\dots ,n-1}(\mathbb C^n)$. To this end, recall the following notations:
\[
\mathcal{O}=\left\{ w=(w_1,\dots,w_{n}) \in \mathbb{C}^{n-1} \times \dots \times \mathbb{C}^{n-1}  \mid w_i^* w_j =-1,\,  1 \le i <j \le n \right\},
\]
and
\[
\mathcal{D}=\left\{ \left(\begin{matrix}
        a_{11} & \dots & a_{1n} \\
        \vdots & \ddots &\vdots \\
        a_{n1} & \dots & a_{nn}
    \end{matrix} \right) \in \mathbf{U}(n) \Bigm| a_{1n} \neq 0, \dots, a_{nn} \neq 0   \right\} .
\]
We will then use the following cylindric parametrization of $\mathcal{D}$ :
\begin{align}
\begin{cases}
\begin{array}{ccccc}
\mathbb{R}^n &\times & \mathcal{O} & \to & \mathcal{D} \\
(\theta & , &w) & \to &   \left(\begin{matrix}
        \frac{e^{i\theta_1}w_{11} }{\sqrt{1+|w_1|^2}} & \dots & \frac{e^{i\theta_n}w_{1n} }{\sqrt{1+|w_n|^2}} \\
        \vdots & \ddots &\vdots \\
        \frac{e^{i\theta_1}w_{(n-1)1} }{\sqrt{1+|w_1|^2}} & \dots & \frac{e^{i\theta_n}w_{(n-1)n} }{\sqrt{1+|w_n|^2}} \\
       \frac{e^{i\theta_1} }{\sqrt{1+|w_1|^2}}  & \dots &  \frac{e^{i\theta_n} }{\sqrt{1+|w_n|^2}} 
    \end{matrix} \right)
\end{array} .
\end{cases}
\end{align}

In this parametrization,  the connection form  \eqref{connection form 3} admits the following decomposition:

 \begin{align}\label{eq-contact-eta-sphere}
\eta_j =  d\theta_j+\frac{i}{2(1+|w_j|^2)}\sum_{k=1}^{n-1}(w_{kj}d\overline{w}_{kj}-\overline{w}_{kj}dw_{kj}).
\end{align}

\subsection{Horizontal Brownian motion on \texorpdfstring{$\mathbf{U}(n)$}{U(n)}}

Recall the Riemannian submersion $p: \mathcal{D} \to \mathcal{O}$ defined by
\begin{align}\label{submersion p:eq2}
p\left(\begin{matrix}
        a_{11} & \dots & a_{1n} \\
        \vdots & \ddots &\vdots \\
        a_{n1} & \dots & a_{nn}
    \end{matrix} \right) 
    =\left( \left(\begin{matrix}
        a_{11}/a_{n1} \\
        \vdots \\
        a_{(n-1)1}/a_{n1}
        \end{matrix}
        \right), \dots, \left(\begin{matrix}
        a_{1n}/a_{nn} \\
        \vdots \\
        a_{(n-1)n}/a_{nn}
        \end{matrix}
        \right) \right)
\end{align}
and the  Laplace-Beltrami operator $\Delta$ on $\mathcal{O}$ from Proposition \eqref{Laplacian full flag}. Using the above Riemannian submersion one can obtain the horizontal Laplacian on $\mathbf{U}(n)$ by taking the horizontal lift of $\Delta$ through the projection map $p$, i.e. as the operator $\Delta_{\mathcal{H}}$ satisfying
\begin{align}\label{horizontal laplacian}
\Delta (f \circ p)=(\Delta_{\mathcal{H}} f) \circ p, \quad f \in C^\infty (\mathcal{D}).
\end{align}

\begin{definition}
A horizontal Brownian motion on $\mathbf{U}(n)$ is a diffusion on $\mathcal{D}$ with generator $\frac{1}{2} \Delta_\mathcal{H}$.
\end{definition}

\begin{definition}\label{def-theta}
 Let $w(t)$ be a Brownian motion on $F_{1,2,\dots ,n-1}(\mathbb C^n)$, i.e. a diffusion process with generator $\frac{1}{2} \Delta$ where $\Delta$  is given by
     \eqref{eq-Del-n-1}. The stochastic area process is the $\mathbb{R}^n$-valued process
     \begin{align}\label{eq-theta}
          \theta(t)=(\theta_1(t),\dots,\theta_n(t)) ,
     \end{align}
where
\[
\theta_j(t):=-\int_{w_j[0,t]} \alpha=\frac{1}{2i}\sum_{k=1}^{n-1} \int_0^t \frac{w_{kj}(s)  d\overline{w}_{kj}(s)-\overline{w}_{kj}(s) dw_{kj}(s)}{1+|w_j(s)|^2},
\]
and the above stochastic integrals are understood in the Stratonovich, or equivalently in the It\^o sense.

\end{definition}

 Here $\alpha$ denotes the area form on $\mathbb{C}P^{n-1}$, which is the one-form given in the local affine coordinates of $\mathbb{C}P^{n-1}$ by
\[
\alpha=\frac{i}{2}\sum_{k=1}^{n-1}  \frac{w_{k}  d\overline{w}_{k}-\overline{w}_{k} dw_{k}}{1+|w|^2}.
\]
The area form $\alpha$ was first introduced in \cite{MR3719061} and we refer to \cite[Section 5.1]{book} for an extensive overview of its properties, the most important one being that $d\alpha$ is almost everywhere the K\"ahler form on $\mathbb{C}P^{n-1}$, hence the terminology area form. Since $w_j(t)$ is a Brownian motion on $\mathbb{C}P^{n-1}$, the process $\theta_j(t)$ is therefore interpreted as a stochastic area process in this space. With respect to  \cite{book, MR3719061} we point out a sign difference in our definition of the stochastic area. We note that the full flag manifold is itself a K\"ahler manifold, even a projective variety, since it is immersed in $\mathbb{C}P^{n-1} \times \cdots \times \mathbb{C}P^{n-1}$ and its K\"ahler form is given on $\mathcal{O}$  by
\[
\omega=\sum_{j=1}^{n} d\alpha_j ,
\]
where 
\begin{equation*}
\alpha_j :=\frac{i}{2(1+|w_j|^2)}\sum_{k=1}^{n-1}(w_{kj}d\overline{w}_{kj}-\overline{w}_{kj}dw_{kj}).
\end{equation*}

\begin{theorem}\label{skew product unitary}
Let $w(t)$ be a Brownian motion on $F_{1,2,\dots ,n-1}(\mathbb C^n)$ and let $\theta(t)$ be its stochastic area process as in \eqref{eq-theta}. The process
\begin{align*}
X(t)=
  \left(\begin{matrix}
        \frac{e^{i\theta_1(t)}w_{11} (t)}{\sqrt{1+|w_1(t)|^2}} & \dots & \frac{e^{i\theta_n(t)}w_{1n}(t) }{\sqrt{1+|w_n(t)|^2}} \\
        \vdots & \ddots &\vdots \\
        \frac{e^{i\theta_1(t)}w_{(n-1)1}(t) }{\sqrt{1+|w_1(t)|^2}} & \dots & \frac{e^{i\theta_n(t)}w_{(n-1)n} (t)}{\sqrt{1+|w_n(t)|^2}} \\
       \frac{e^{i\theta_1(t)} }{\sqrt{1+|w_1(t)|^2}}  & \dots &  \frac{e^{i\theta_n(t)} }{\sqrt{1+|w_n(t)|^2}} 
    \end{matrix} \right)
\end{align*}
is a horizontal Brownian motion on the unitary group $\mathbf{U}(n)$.
\end{theorem}

\begin{proof}
To prove that $X$ is a horizontal Brownian motion, one needs to prove the following two properties (see \cite[Theorem 3.1.10]{book}): 
\begin{itemize}
\item[(i)]  it projects down to the Brownian motion on 
$F_{1,2,\dots ,n-1}(\mathbb C^n)$; 
\item[(ii)]  it is a horizontal process. 
\end{itemize}
The first property follows directly from the definition of the Riemannian submersion 
$p$. By applying equation \eqref{submersion p:eq2}, we obtain:
\[
p(X(t))=(w_1(t), \dots, w_n(t)).
\]
As for the second property, it follows from the decomposition \eqref{eq-contact-eta-sphere}:
\begin{align*}
\int_{X[0,t]} \eta_j&=\theta_j(t)+\frac{i}{2}\int_{X[0,t]}\frac{\sum_{k=1}^{n-1}(w_{kj}d\overline{w}_{kj}-\overline{w}_{kj}dw_{kj})}{(1+|w_j|^2)} \\
 &=\theta_j(t)-\frac{1}{2i}\sum_{k=1}^{n-1} \int_0^t \frac{w_{kj}(s)  d\overline{w}_{kj}(s)-\overline{w}_{kj}(s) dw_{kj}(s)}{1+|w_j(s)|^2} \\
  &= 0,
\end{align*}
where the last equality holds by the definition of the process $(\theta_j(t))_{t\geq 0}$.
\end{proof}

\begin{theorem}\label{horizontal Laplace}
 Let $w(t)$ be a Brownian motion on $F_{1,2,\dots ,n-1}(\mathbb C^n)$, and let $\theta(t)$ be its stochastic area process as defined in Definition \ref{def-theta}.  The process $(w(t),\theta(t))$ is a diffusion with generator
    \begin{align*}
        &2 \sum_{j=1}^n\sum_{p,q=1}^{n-1}(\delta_{pq} +w_{pj}\overline{w}_{qj})(1+|w_j|^2)
        \frac{\partial^2}{\partial w_{pj}\partial\overline{w}_{qj}}
        +i\sum_{j=1}^{n}\sum_{p=1}^{n-1}(1+|w_j|^2 )\left( \overline{w}_{pj}\frac{\partial^2}{\partial\overline{w}_{pj}\partial\theta_j} -w_{pj}\frac{\partial^2}{\partial w_{pj} \partial\theta_j}\right) \\
        -& \sum_{p,q=1}^{n-1}\sum_{1\leq j\neq m\leq n} \left\{
        (w_{pm}-w_{pj})(w_{qj}-w_{qm})\frac{\partial^2}{\partial w_{pj}\partial w_{qm}}+(\overline{w}_{pm}-\overline{w}_{pj})(\overline{w}_{qj}-\overline{w}_{qm})\frac{\partial^2}{\partial \overline{w}_{pj}\partial\overline{w}_{qm}} \right\}
        \\
        +&i\sum_{1\leq j\neq m\leq n}\sum_{k=1}^{n-1}\left((w_{kj}-w_{km})\frac{\partial^2}{\partial w_{km}\partial\theta_j}
        -(\overline{w}_{kj}-\overline{w}_{km})\frac{\partial^2}{\partial \overline{w}_{km}\partial\theta_j}\right)
        + \frac{1}{2}\sum_{j=1}^{n}|w_j|^2\frac{\partial^2}{\partial\theta^2_j} +\frac{1}{2}\sum_{1\le j\neq m\le n}\frac{\partial^2}{\partial\theta_j\partial\theta_m} . 
    \end{align*}
\end{theorem}
\begin{proof}
Using Lemma \ref{lemma-qv-dw} and the formula
\[
d\theta_j(t)=\frac{1}{2i}\sum_{k=1}^{n-1} \frac{w_{kj}(t)  d\overline{w}_{kj}(t)-\overline{w}_{kj}(t) dw_{kj}(t)}{1+|w_j(t)|^2}
\]
one can compute the quadratic covariations $dw_{kj} d\theta_\ell, d\bar{w}_{kj} d\theta_\ell$ and $d\theta_\ell d\theta_m$.
These are given by
\begin{align*}
    d\theta_j d\theta_m 
    =&\frac{-1}{4(1+|w_j|^2)(1+|w_m|^2)}\left(\sum_{\ell =1}^{n-1} (w_{\ell j}d\overline{w}_{\ell j} -\overline{w}_{\ell j}dw_{\ell j})\right)\left(\sum_{k=1}^{n-1} (w_{km}d\overline{w}_{km} -\overline{w}_{km}dw_{km})\right) \\
    =&\frac{1}{2(1+|w_j|^2)(1+|w_m|^2)}\big((w_j\cdot\overline{w}_m -|w_j|^2)(w_m\cdot\overline{w}_j-|w_m|^2) \\
    &\qquad+2\delta_{jm}(|w_j|^2 +|w_j|^4)(1+|w_m|^2)
    +(\overline{w}_j\cdot w_m -|w_j|^2)(w_m\cdot w_j-|w_m|^2)\big) dt
    \\
    =&(1-\delta_{jm}+\delta_{jm}|w_j|^2) dt
\end{align*} 
and
\begin{align*}
    d\theta_j dw_{k m} =&\frac{1}{2i (1+|w_j|^2)}\sum_{\ell =1}^{n-1} (w_{\ell j}d\overline{w}_{\ell j}dw_{k m} -\overline{w}_{\ell j}dw_{\ell j}dw_{k m}) \\
    =&\frac{2}{2i (1+|w_j|^2)}\sum_{\ell =1}^{n-1} (\delta_{mj}w_{\ell j}(1+|w_m|^2)(\delta_{k\ell } +w_{km}\overline{w}_{\ell m}) +\overline{w}_{\ell j}(w_{\ell m} -w_{\ell j})(w_{kj} -w_{km}))dt \\
    =&\frac{\delta_{jm}}{i}w_{k j}(1 +|\overline{w}_{j}|^2 )dt-\frac{1}{i}(w_{kj} -w_{km})dt.
\end{align*}
This yields the stated formula for the generator.
\end{proof}

\begin{remark}
Since the process \((X(t))_{t \ge 0}\) in Theorem~\ref{skew product unitary} is a horizontal Brownian motion, its generator is \(\frac{1}{2} \Delta_\mathcal{H}\), where \(\Delta_\mathcal{H}\) denotes the horizontal Laplacian defined in~\eqref{horizontal laplacian}. Consequently, the generator computed in Theorem~\ref{horizontal Laplace} represents the expression of \(\frac{1}{2} \Delta_\mathcal{H}\) in the cylindrical parametrization. Moreover, since \(\Delta_\mathcal{H}\) is the horizontal lift of the Laplace--Beltrami operator \(\Delta\) on \(\mathcal{O}\), the generator in Theorem~\ref{horizontal Laplace} can also be derived by lifting the formula obtained in Proposition~\ref{Laplacian full flag}, following an approach similar to that of~\cite[Theorem~5.1.6]{book}.
\end{remark}

\begin{corollary}\label{skew product}
Let $(w(t),\theta(t))$ be the diffusion process  as in Theorem \ref{horizontal Laplace}.  Let $\lambda(t)$ be as defined in \eqref{eq-lambda}.
Then the joint process $(\lambda(t),\theta(t))$ is a diffusion with generator
 \begin{align*}
  & 2 \sum_{j=1}^{n} \lambda_j(1-\lambda_j) \frac{\partial^2}{\partial \lambda_{j}^2} +2 \sum_{j=1}^{n}(1-n\lambda_j) \frac{\partial}{\partial \lambda_{j}}
 -2\sum_{1 \le j\not=\ell \le n}\lambda_j\lambda_{\ell}  
\frac{\partial^2}{\partial \lambda_{j}\partial \lambda_\ell}  \\
+&\frac{1}{2}\sum_{j=1}^{n}\frac{1-\lambda_j}{\lambda_j} \frac{\partial^2}{\partial\theta_j^2} +\frac{1}{2}\sum_{1 \le i\neq j \le n}\frac{\partial^2}{\partial\theta_i\partial\theta_j} .
\end{align*}

Therefore, for every $t > 0$, conditionally on $(\lambda(s), s \leq t)$, the random vector $\theta(t)$ is Gaussian with mean zero and covariance matrix
    \begin{align}\label{covariance}
       \Sigma (t)=\begin{pmatrix}
            \int_0^t\frac{(1-\lambda_1(s))}{\lambda_1(s)}ds  &\dots & t\\
            \vdots & \ddots & \vdots\\
            t &\dots & \int_0^t\frac{(1-\lambda_n(s))}{\lambda_{n}(s)}ds
        \end{pmatrix}
    \end{align}
\end{corollary}

\begin{proof}
We use the formulas
    \begin{align*}
        \frac{\partial r_j}{\partial w_{pj}} =\frac{1}{2}\frac{\overline{w}_{pj}}{r_j} ,
        \quad \frac{\partial r_j}{\partial \overline{w}_{pj}} =\frac{1}{2}\frac{w_{pj}}{r_j}
    \end{align*}
   and apply the chain rule to smooth  radial functions 
   $f:F_{1,\dots ,n-1}(\mathbb{C}^n)\rightarrow\mathbb{R}$. Doing so, we get for any $1 \leq p,q \leq n-1$ and any $1 \leq j\leq n$:
    \begin{align*}
        \frac{\partial^2}{\partial w_{pj}\partial\overline{w}_{qj}}f(r_1,\dots ,r_{n}) 
        =&\frac{\partial}{\partial w_{pj}}\left(\frac{\partial f}{\partial r_j}\frac{w_{qj}}{2r_j}\right)
        =\left(\frac{\delta_{pq}}{2r_j} 
        -\frac{w_{q j}\overline{w}_{pj}}{4r_j^3}\right)\frac{\partial f}{\partial r_j} +\frac{w_{qj}\overline{w}_{pj}}{4r_j^2}
        \frac{\partial^2 f}{\partial r_j^2}
        \\
    \end{align*}
    and similarly for any $1 \leq j, m \leq n$, 
    \begin{align*}
        \frac{\partial^2}{\partial w_{pj}\partial w_{q m}}f(r_1,\dots ,r_{n}) 
        =\delta_{jm}\left(\frac{\delta_{pq}}{2r_j} -\frac{\overline{w}_{q m}\overline{w}_{pj}}{4r_j^3}\right)\frac{\partial f}{\partial r_j} +\frac{\overline{w}_{q m}\overline{w}_{pj}}{4r_mr_j}\frac{\partial^2 f}{\partial r_j\partial r_m} .
    \end{align*}

Using the relation \eqref{eq-orth-w}, we see that the operator in Theorem \ref{horizontal Laplace} acts on functions depending only on $(r_1,\dots,r_n,\theta_1,\dots,\theta_n)$ as $1/2$ of the operator
    \begin{align*}
    &\sum_{j=1}^{n}(1+r_j^2)^2\frac{\partial^2}{\partial r_j^2}
    +\sum_{j=1}^{n}\frac{1+r_j^2}{r_j}\left( 2n-3+r_j^2\right)\frac{\partial}{\partial r_j}
    -\sum_{1 \le j\neq m\le n}\frac{(1+r_j^2)(1+r_m^2)}{r_jr_m}\frac{\partial^2}{\partial r_j\partial r_m} \\
    +&\sum_{j=1}^{n}r_j^2\frac{\partial^2}{\partial\theta_j^2} +\sum_{1 \le i\neq j \le n}\frac{\partial^2}{\partial\theta_i\partial\theta_j} .
\end{align*}
Performing the change of variables, 
\begin{equation*}
\lambda_j = \frac{1}{1+r_j^2} \, \Leftrightarrow \, r_j^2 = \frac{1-\lambda_j}{\lambda_j}, \quad 1 \leq j \leq n,     
\end{equation*}
we obtain the first claim of the corollary.
The second claim then follows as a direct consequence.  
\end{proof}

\subsection{Characteristic function of the stochastic area process}
In this section we study the stochastic area processes on the full flag manifold $F_{1,2,\dots ,n-1}(\mathbb C^n)$ and compute the characteristic function of their joint distributions. We begin by recalling the simplex
\begin{align*}
    \mathcal{T}_n :=\{\lambda\in\mathbb{R}^{n}\mid\lambda_j\geq 0, 1\leq j\leq n,\quad\lambda_1 +\dots +\lambda_n =1\}. 
\end{align*}

\begin{theorem}\label{carac area flag}
   Let $(w(t),\theta(t))$ be the diffusion process  as in Theorem \ref{horizontal Laplace}, and $\lambda(t)$ be as given in Corollary \ref{skew product}. For any $u=(u_1 ,\dots ,u_{n})\in \mathbb{R}^n$, any $\lambda(0), \lambda$ in the interior of $\mathcal{T}_{n}$, and any $t>0$, we have
    \begin{align*}
         & \mathbb{E}\left(e^{i\sum_{j=1}^{n} u_j\theta_j(t)}\mid \lambda (t)=\lambda \right) = e^{-(n-1)  \sum_{j=1}^{n}|u_j| t -\frac{1}{2}\sum_{1\leq j\neq m\leq n} (u_j u_m +|u_j u_m|)t } \\
         &\quad\quad\cdot \prod_{j=1}^n\left(\frac{\lambda_j(0)}{\lambda_j}\right)^{\frac{|u_j|}{2}}
         \quad\frac{q^{(1/2+|u_1|,\dots,1/2+|u_n|)}_{2t}(\lambda^{(n-1)}(0),\lambda^{(n-1)})}{q^{(1/2,\dots,1/2)}_{2t}(\lambda^{(n-1)}(0),\lambda^{(n-1)})} \frac{W^{(1/2+|u_1|,\dots,1/2+|u_n|)}(\lambda^{(n-1)}) }{W^{(\frac{1}{2},\dots, \frac{1}{2})}(\lambda^{(n-1)}) } ,
    \end{align*}
    where $q_t^{(\kappa_1 ,\dots ,\kappa_{n})}$ and
    $W^{(\kappa_1 ,\dots ,\kappa_{n})}$ are given by 
    \eqref{kernel jacobi simplex} and \eqref{DenDirich} respectively and $\lambda^{(n-1)}=(\lambda_1,\cdots,\lambda_{n-1})$.
\end{theorem}
\begin{proof}
From Corollary \ref{skew product} we know that conditioned on 
$(\lambda (s), s \leq t)$ the winding  $ \theta(t) =
(\theta_1(t) ,\dots ,\theta_{n}(t))$ is a  Gaussian variable with mean zero and covariance matrix $\Sigma(t)$ given by \eqref{covariance}. It  follows that:
\begin{align*}
        \mathbb{E}\left( e^{i\sum_{j=1}^n u_j \theta_j(t)}\mid\lambda (t)=\lambda\right) =&\mathbb{E}\left(e^{-\frac{1}{2} u^T\Sigma(t) u}
        \mid\lambda (t) =\lambda \right) \\
        =&\mathbb{E}\left(\exp\left(-\frac{1}{2}\sum_{j=1}^{n} u_j^2\int_0^t\frac{1-\lambda_j (s)}{\lambda_j (s)}ds -\frac{1}{2}\sum_{1\leq j\neq m\leq n} u_j u_m t\right)\biggm|\lambda (t) =\lambda\right) .
    \end{align*}
It only remains to derive the expression of 
\begin{equation}\label{eq-mid-goal}
    \mathbb{E}\left(\exp\left(-\frac{1}{2}\sum_{j=1}^{n} u_j^2\int_0^t\frac{1-\lambda_j (s)}{\lambda_j (s)}ds \right)\bigg|\lambda (t) =\lambda\right).
    \end{equation}
To this end, we define the following function on $\mathcal T_n$:
\begin{align*}
    f(\lambda_1 ,\dots ,\lambda_{n}) :=\prod_{j=1}^{n}\lambda_j^{\frac{|u_j|}{2}}.
\end{align*}
Let $\widehat{\mathcal{G}}_{1/2,\dots,1/2}$ be a Jacobi operator on $\mathcal T_n$ as in \eqref{GenJacSim1}, with $\kappa=(1/2,\dots,1/2)$. Applying $\widehat{\mathcal{G}}_{1/2,\dots,1/2}$ to $f$ we obtain
\begin{align*}
    \widehat{\mathcal{G}}_{1/2,\dots,1/2} f=&\sum_{j=1}^{n} \lambda_j(1-\lambda_j)\frac{\partial^2 f}{\partial\lambda_{j}^2} 
    + \sum_{j=1}^{n}\left( 1 
    - n\lambda_j\right)\frac{\partial f}{\partial\lambda_j}
    -\sum_{1 \leq j \neq \ell \leq n}\lambda_j\lambda_{\ell} \frac{\partial^2 f}{\partial\lambda_{j}\lambda_{\ell}} \\
    =&\sum_{j=1}^{n}\frac{|u_j|}{2}\left(\frac{|u_j|}{2} -1\right)\frac{1-\lambda_j }{\lambda_j} f 
    + \sum_{j=1}^{n}\frac{|u_j|}{2}\left( \frac{1}{\lambda_j} 
    - n\right) f
    -\sum_{1 \leq j \neq \ell \leq n}\frac{|u_j u_{\ell} |}{4} f,
\end{align*}
which implies that $f$ is an eigenfunction of the operator 
\begin{equation*}
\widehat{\mathcal{G}}_{1/2,\dots,1/2} -\frac{1}{4}\sum_{j=1}^{n}u_j^2\frac{1-\lambda_j}{\lambda_j}    
\end{equation*} 
associated with the eigenvalue
\begin{align*}
    -\frac{(n-1)}{2}\sum_{j=1}^{n}|u_j| -\frac{1}{4}\sum_{1\leq i\neq j\leq n}|u_i u_j| .
\end{align*}
From Theorem \ref{Th1} we know that $(\lambda (t))_{t\geq 0}$ is a Jacobi process in the simplex $\mathcal{T}_n$ with generator $2\widehat{\mathcal{G}}_{1/2,\dots,1/2}$, It\^o's formula shows that the process
    \begin{align*}
        D_t^{u} :=e^{(n-1)\sum_{j=1}^{n}|u_j| t +\frac{1}{2}\sum_{1\leq j\neq m\leq n} |u_ju_m| t}\left(\prod_{j=1}^n\left(\frac{\lambda_j(t)}{\lambda_j(0)}\right)^{\frac{|u_j|}{2}} e^{-\frac{u^2_j}{2}\int_0^t\frac{1-\lambda_j (s)}{\lambda_j (s)}ds}\right)
    \end{align*}
is a local martingale. One can easily verify that $$D_t^{u}\leq \frac{e^{(n-1)\sum_{j=1}^{n} |u_j| t +\frac{1}{2}\sum_{j\neq m} |u_ju_m| t}}{\prod_{j=1}^n\lambda_j(0)^{\frac{|u_j|}{2}}},$$ which implies that $ D_t^{u}$ is in fact a martingale.
As a matter of fact, we may define a new probability measure $\mathbb{P}^{u}$ by setting for any $t > 0$, $d\mathbb{P}^{u} :=D_t^{u}d\mathbb{P}$. 
We then have for every bounded Borel function $F$ that
    \begin{align*}
      &  \mathbb{E}\left( F(\lambda_1(t),\dots ,\lambda_{n}(t))e^{-\frac{1}{2}\sum_{j=1}^{n} u^2_j\int_0^t\frac{1-\lambda_j(s)}{\lambda_j (s)}ds}\right)\\
        =&e^{-(n-1)\sum_{j=1}^{n}|u_j| t -\frac{1}{2}\sum_{1\leq j\neq m\leq n}|u_ju_m| t} \prod_{j=1}^n\lambda_j(0)^{\frac{|u_j|}{2}}\mathbb{E}^{u}\left(\frac{F(\lambda_1(t),\dots ,\lambda_{n}(t))}{\prod_{j=1}^{n}\lambda_j (t)^{\frac{|u_j|}{2}}}\right) .
    \end{align*}
 Let  $s_t^{(u)}(\lambda(0),d\lambda)$ denote the probability distribution of $\lambda(t)$ 
under $\mathbb{P}^u$, and  $\widehat{q}^{1/2,\dots,1/2}_{2t}(\lambda(0),d\lambda)$ the probability distribution of $\lambda(t)$ under $\mathbb{P}$.  The above equality then implies
\begin{align}\label{law change}
& \mathbb{E}\left( e^{-\frac{1}{2}\sum_{j=1}^{n} u^2_j\int_0^t\frac{1-\lambda_j(s)}{\lambda_j (s)}ds} \biggm|\lambda (t) =\lambda\right)\widehat{q}^{(1/2,\dots,1/2)}_{2t}(\lambda(0),d\lambda) \notag \\
= &e^{-(n-1)\sum_{j=1}^{n}|u_j| t -\frac{1}{2}\sum_{1\leq j\neq m\leq n}|u_ju_m| t} \prod_{j=1}^n\left(\frac{\lambda_j(0)}{\lambda_j}\right)^{\frac{|u_j|}{2}} s_t^{(u)}(\lambda(0),d\lambda).
\end{align}
Comparing to \eqref{eq-mid-goal} we are left to compute $s_t^u(\lambda(0),d\lambda)$, using the Girsanov theorem. Recall the stochastic differential equation satisfied by $\lambda(t)$ as in Theorem \ref{Th1}.
    By It\^o's formula we have
\begin{align*}
    d\ln (\lambda_j (t)) =&\frac{d\lambda_j(t)}{\lambda_j(t)} -\frac{1}{2}\frac{d\lambda_j(t)d\lambda_j (t)}{\lambda_j(t)^2}\\
    =&\frac{2}{\sqrt{\lambda_j(t)}}\sum_{\ell=1, \ell\neq j}^{n}d\gamma_{\ell j}(t)\sqrt{\lambda_{\ell}(t)} +2\left(\frac{1}{\lambda_j(t)}-n\right) dt -\frac{2}{\lambda_j(t)}\sum_{\ell =1,\ell\neq j}^n\lambda_{\ell} (t)dt .
\end{align*}
 Plugging in $1-\lambda_j = \sum_{l\neq j}\lambda_l$ we then obtain that   
 \begin{align*}
   d\ln (\lambda_j (t)) =\frac{2}{\sqrt{\lambda_j(t)}}\sum_{\ell=1, \ell\neq j}^{n}d\gamma_{\ell j}(t)\sqrt{\lambda_{\ell}(t)} -2(n-1)dt.
    \end{align*}
    This gives
    \begin{align*}
    \prod_{j=1}^n\lambda_j(t)^{\frac{|u_j|}{2}}=\exp \left( \sum_{j=1}^n\sum_{\ell=1, \ell\neq j}^{n} \frac{|u_j|}{\sqrt{\lambda_j(t)}}\int_0^t  \sqrt{\lambda_{\ell}(s)} d\gamma_{\ell j}(s)  \right) \exp \left( -(n-1)\left(\sum_{i=1}^n |u_i|\right)t\right).
    \end{align*}
Now define the stochastic processes $\tilde{\gamma}_{\ell j}(t)$ by
 \begin{align*}
        d\tilde{\gamma}_{\ell j}(t) :=d\gamma_{\ell j}(t) -\Theta_{\ell j}(t) dt 
    \end{align*}
for $\ell<j$ and set $\tilde{\gamma}_{j\ell}(t)=-\tilde{\gamma}_{\ell j}(t)$, where
\begin{align*}
    \Theta_{\ell j}(t) :=
        \frac{|u_j|}{\sqrt{\lambda_j(t)}}\sqrt{\lambda_{\ell} (t)} -\frac{|u_{\ell}|}{\sqrt{\lambda_{\ell} (t)}}\sqrt{\lambda_{j} (t)} . 
\end{align*}
By Girsanov's theorem the process $\tilde{\gamma}_{\ell j}(t)$    
is a Brownian motion under $\mathbb{P}^{u}$.
This means that under $\mathbb{P}^{u}$, $(\lambda_1(t),\dots ,\lambda_{n}(t))_{t\geq 0}$ satisfies the stochastic differential equation
\begin{align*}
     d\lambda_j(t) =&2\sqrt{\lambda_j(t)}\sum_{\ell=1, \ell\neq j}^{n}d\tilde{\gamma}_{\ell j}(t)\sqrt{\lambda_{\ell} (t)} 
    +2(1-n\lambda_j(t))dt\\
    +&2\sum_{1\leq\ell <j}(|u_{j}|\lambda_{\ell} (t) -|u_{\ell}|\lambda_{j} (t))dt
    -2\sum_{ j <\ell\leq n}(|u_{\ell}|\lambda_j (t) -|u_{j}|\lambda_{\ell} (t))dt\\
    =&2\sqrt{\lambda_j(t)}\sum_{\ell=1, \ell\neq j}^{n}d\tilde{\gamma}_{\ell j}(t)\sqrt{\lambda_{\ell}(t)} 
    +2(1-n\lambda_j(t) -|u|\lambda_j (t) +|u_j|)dt ,
\end{align*}
where we used that $\sum_{\ell =1}^n\lambda_j=1$ and the notation $|u|:=\sum_{i=1}^n |u_i|$.
In particular, the generator of $\lambda(t)$ under $\mathbb{P}^u$ is given by a Jacobi operator on $\mathcal{T}_{n}$:
\begin{align*}
    &2\sum_{j=1}^{n} \lambda_j(1-\lambda_j)\frac{\partial^2}{\partial\lambda_{j}^2} 
        +2\sum_{j=1}^{n}\left[\left(1+|u_j|\right) 
        - \left( n+|u|\right)\lambda_j\right]\frac{\partial}{\partial\lambda_j} -2\sum_{1\leq j\neq\ell\leq n}\lambda_j\lambda_{\ell} \frac{\partial^2}{\partial\lambda_{j}\lambda_{\ell}}
        .
\end{align*}
The conclusion then follows from  \eqref{eq-mid-goal}, \eqref{law change} and Section \ref{section Jacobi}.
\end{proof}

\subsection{Limit theorem}

We are now ready to prove the limit theorem for the asymptotics of the stochastic area.

\begin{theorem}\label{limit stochastic area}
Let $(w(t),\theta(t))$ be the diffusion process  as in Theorem \ref{horizontal Laplace}. Then the following convergence holds in distribution
 \[
 \frac{\theta(t)}{t} \to \left( C^1_{n-1},\dots,C^n_{n-1}\right)\textrm{ as }t \to +\infty ,
 \]
 where  $C^1_{n-1},\dots,C^n_{n-1}$ are independent Cauchy random variables with parameter $n-1$.
\end{theorem}

\begin{proof}
Let $u_1,\dots,u_n \in \mathbb{R}$. From Theorem \ref{carac area flag} one has

\begin{align*}
         & \mathbb{E}\left(e^{i\sum_{j=1}^{n} u_j\theta_j(t)} \right)
        = e^{-(n-1)  \sum_{j=1}^{n}|u_j| t -\frac{1}{2}\sum_{1\leq j\neq m\leq n} (u_j u_m +|u_j u_m|)t }  \\
        &\qquad \int_{\mathcal{T}_n} \prod_{j=1}^n\left(\frac{\lambda_j(0)}{\lambda_j}\right)^{\frac{|u_j|}{2}} \frac{q^{(1/2+|u_1|,\dots,1/2+|u_n|)}_{2t}(\lambda^{(n-1)}(0),\lambda^{(n-1)})}{q^{(1/2,\dots,1/2)}_{2t}(\lambda^{(n-1)}(0),\lambda^{(n-1)})} \frac{W^{(1/2+|u_1|,\dots,1/2+|u_n|)}(\lambda^{(n-1)}) }{W^{(\frac{1}{2},\dots, \frac{1}{2})}(\lambda^{n-1}) } d\mathbb{P}_{\lambda(t)}(\lambda) ,
    \end{align*}
where $\mathbb{P}_{\lambda(t)}$ is the law of $\lambda(t)$.
    It then follows from \eqref{kernel jacobi simplex}  and dominated convergence that

\[
\lim_{t \to +\infty} \mathbb{E}\left(e^{i\sum_{j=1}^{n} u_j\frac{\theta_j(t)}{t}} \right) =e^{-(n-1)\sum_{j=1}^{n}|u_j|  }.
\]
\end{proof}

\section{Simultaneous Brownian windings on the complex sphere}\label{sec:Brownian-winding-on-spheres}

In this final section, we present an application of our previous results to the study of simultaneous Brownian windings on spheres. Consider the $2n-1$ dimensional sphere $\mathbb{S}^{2n-1} \subset \mathbb{C}^n$ and a  Brownian motion 
\[
X(t)=(X_1(t),\dots,X_n(t)),\quad t\ge0
\]
on it. Assuming that $X_j(0) \neq 0$, $1 \le j \le n$, we can then consider the polar decompositions
\[
X_j(t)=\varrho_j(t) e^{i \eta_j(t)}, \quad 1 \le i \le n,
\]
where $\varrho_j(t)$ and $\eta_j(t)$ are continuous and real-valued processes with $\varrho_j(0)>0$.  Our goal will be to understand the joint distribution of the winding process
\[
(\eta_1(t),\dots,\eta_n(t))
\]
and to study its asymptotics as $t \to +\infty$.  Since our methods yield more general results, we will work with a general class of diffusion processes on $\mathbb{S}^{2n-1}$, which include Brownian motion as a special case. 

Our framework is the following. The group 
\[
\mathbf{U}(1)^n =\left\{ (e^{i\theta_1},\dots, e^{i\theta_n}) \mid \theta_i \in \mathbb{R} \right\}
\]
isometrically acts on $\mathbb{S}^{2n-1}$ as
\begin{align}\label{action diagonal to spehere}
(e^{i\theta_1},\dots, e^{i\theta_n}) \dot (z_1,\dots,z_n)=(e^{i\theta_1} z_1,\dots,e^{i\theta_n}z_n).
\end{align}
This yields a fibration of homogeneous spaces
\begin{align}\label{diagonal fibration}
    \mathbf{U}(1)^n \rightarrow \mathbb{S}^{2n-1} \rightarrow \frac{\mathbf U(n)}{\mathbf U(n-1)\times \mathbf U(1)^n} ,
\end{align}
where the map $\mathbb{S}^{2n-1} \rightarrow \frac{\mathbf U(n)}{\mathbf U(n-1)\times \mathbf U(1)^n}$ is a Riemannian submersion with totally geodesic fibers isometric to the torus $\mathbf{U}(1)^n$. Consider then the generators $\frac{\partial}{\partial \theta_1},\dots,\frac{\partial}{\partial \theta_n}$ of the group action \eqref{action diagonal to spehere}; those are Killing vector fields on $\mathbb{S}^{2n-1}$. Due to the fibration \eqref{diagonal fibration} the standard Riemannian metric on $\mathbb{S}^{2n-1}$ can be orthogonally decomposed as
\[
g_{\mathbb{S}^{2n-1}}=g_{\mathcal{H} }\oplus  g_{\mathcal V} ,
\]
where $\mathcal{V}$ is the sub-bundle spanned by $\frac{\partial}{\partial \theta_1},\dots,\frac{\partial}{\partial \theta_n}$ and $\mathcal{H}$ is  its orthogonal complement. Given
\[
\mu=(\mu_1,\dots,\mu_n)
\]
with $\mu_i>0$, we will consider a new Riemannian metric on $\mathbb{S}^{2n-1}$ given by
\[
\begin{cases}
g_\mu(X,Y)=g(X,Y), \quad X,Y \in \mathcal{H} \\
g_\mu \left(\frac{\partial}{\partial \theta_i},\frac{\partial}{\partial \theta_j}\right)=\frac{\delta_{ij}}{\mu_i \mu_j}, \quad 1\le i,j \le n,
\end{cases}
\]
where $\delta_{ij}$ is the Kronecker symbol which is $1$ if $i=j$ and $0$ otherwise. Of course, $g_\mu=g$ if all the $\mu_i$'s are $1$. We will denote by $\mathbb{S}^{2n-1}_\mu$ the sphere $\mathbb{S}^{2n-1}$ when we want to stress that it is equipped with the Riemannian metric $g_\mu$.

We have the following theorem that allows us to relate the simultaneous windings of Brownian motions on spheres to the geometry of the full flag manifold.

\begin{theorem}\label{Representation BM sphere}
    Let $\mu \in \mathbb{R}_{>0}^n$ and let $(w(t))_{t\geq 0}$ be a Brownian motion on the complex full flag manifold $F_{1,2,\dots ,n-1}(\mathbb C^n)$ with stochastic area process $(\theta(t))_{t\geq 0}$. Let $(\beta(t))_{t \ge 0}$ be a Brownian motion on $\mathbb{R}^n$, which is independent of $(w(t))_{t\geq 0}$. The $\mathbb{S}^{2n-1}$-valued process 
    \begin{align*}
        X_\mu (t):=\begin{pmatrix}\frac{e^{i(\mu_1 \beta_1(t) +\theta_1(t))}}{\sqrt{1+|w_1(t)|^2}}, & \dots, & \frac{e^{i(\mu_n \beta_n(t) +\theta_n(t))}}{\sqrt{1+|w_n(t)|^2}}\end{pmatrix}
    \end{align*}
    is a Brownian motion on $\mathbb S_\mu^{2n-1}$.
\end{theorem}
\begin{proof}
Consider the following commutative diagram:
\[\begin{tikzcd}
	{\mathbf{U}(n)} && {F_{1,2,\dots ,n-1}(\mathbb C^n)} \\
	\\
	{\mathbb{S}^{2n-1}} && {\frac{\mathbf{U}(n)}{\mathbf{U}(n-1)\times \mathbf{U}(1)^n}}
	\arrow[from=1-1, to=1-3]
	\arrow[from=1-1, to=3-1]
	\arrow[from=1-3, to=3-3]
	\arrow[from=3-1, to=3-3]
\end{tikzcd} \]
where the map $\mathbf{U}(n)\to \mathbb{S}^{2n-1}$ is the last row projection, i.e. a unitary matrix is sent to its last row vector.  From this diagram and Theorem \ref{skew product unitary}, it follows that the horizontal Brownian motion of the fibration \eqref{diagonal fibration} is given by
\begin{align*}
\begin{pmatrix}\frac{e^{i\theta_1(t)}}{\sqrt{1+|w_1(t)|^2}}, & \dots, & \frac{e^{i \theta_n(t)}}{\sqrt{1+|w_n(t)|^2}}\end{pmatrix} .
\end{align*}
Now, the Laplace-Beltrami operator on $\mathbb{S}^{2n-1}_\mu$ is 
    \begin{align*}
        \Delta_{\mathbb{S}^{2n-1}_\mu} =\Delta_{\mathcal{H}} +\sum_{j=1}^n \mu_j^2 \frac{\partial^2}{\partial\theta_j^2 } ,
    \end{align*}
    where $\Delta_{\mathcal{H}}$ denotes here the horizontal Laplacian of the fibration \eqref{diagonal fibration}.
    We then note that $\Delta_{\mathcal{H}}$ and $\sum_{j=1}^n \mu_j^2 \frac{\partial^2}{\partial\theta_j^2 }$ commute since the submersion $\mathbb{S}_\mu^{2n-1} \rightarrow \frac{\mathbf U(n)}{\mathbf U(n-1)\times \mathbf U(1)^n}$ is totally geodesic, see \cite[Theorem 4.1.18]{book}.
    The conclusion follows.
    \end{proof}

We obtain the following corollary.

\begin{corollary}\label{limit simul wind}
    Let $\mu \in \mathbb{R}_{>0}^n$ and let $(X_\mu(t))_{t\geq 0}=(X^1_\mu(t),\dots,X^n_\mu(t))_{t\geq 0}$ be a Brownian motion on $\mathbb{S}^{2n-1}_\mu$ such that $X^j_\mu(0) \neq 0$, $1 \le j \le n$. Let $(\eta^1_\mu(t),\dots,\eta^n_\mu(t))_{t \ge 0}$ be a continuous stochastic process such that
    \[
    X^j_\mu(t)=| X^j_\mu(t) | e^{i\eta^j_\mu(t) }.
    \]
    Then, the following convergence 
    \[
  \frac{1}{t}  (\eta^1_\mu(t),\dots,\eta^n_\mu(t))\to \left( C^1_{n-1},\dots,C^n_{n-1}\right) 
 \]
 holds in distribution when $t \to +\infty$, where $C^1_{n-1},\dots,C^n_{n-1}$ are independent Cauchy random variables with parameter $n-1$.
\end{corollary}

\begin{proof}
From the Theorem \ref{Representation BM sphere}, one has in distribution
\[
(\eta^1_\mu(t),\dots,\eta^n_\mu(t))=\left( \mu_1 \beta_1(t) +\theta_1(t), \dots, \mu_n \beta_n(t) +\theta_n(t) \right),
\]
where $\theta(t)$ is the stochastic area process of a Brownian motion on the complex full flag manifold and $\beta(t)$ is an independent Brownian motion. Since $\frac{1}{t} \beta(t)$ almost surely converges to $0$ when $t \to +\infty$, the result follows from Theorem \ref{limit stochastic area}.
\end{proof}

 Interestingly, one can deduce from Corollary \ref{limit simul wind} a limit theorem about some functionals of the Euclidean Brownian motion, and recover the celebrated Spitzer theorem as a corollary.
    Indeed, when $\mu=(1,\dots,1)$, $(X_\mu(t))_{t\geq 0}$ is simply a Brownian motion on the sphere $\mathbb{S}^{2n-1}$ equipped with its standard metric. In that case, the law of $(\eta^1_\mu(t),\dots,\eta^n_\mu(t))$ can also be expressed in terms of a Brownian motion in $\mathbb{C}^n$. Namely, let
    \[
    Z(t)=(Z_1(t),\dots,Z_n(t))
    \]
    be a Brownian motion in $\mathbb{C}^n$ such that $Z_j(0) \neq 0$, $1 \le j \le n$. Then, it is well-known \cite[Example 2.4.1]{book} that the following skew product decomposition holds
    \[
    Z(t)= |Z (t)| X \left( \int_0^t \frac{ds}{|Z(s)|^2} \right) ,
    \]
    where $X$ is a Brownian motion on $\mathbb{S}^{2n-1}$ which is independent from the process $|Z |$. As a consequence
    \[
    \eta \left(\int_0^t \frac{ds}{|Z(s)|^2}\right) =\eta (0)+\zeta \left( t \right) ,
    \]
    where $\zeta$ is the simultaneous winding process of $Z$, i.e
    \[
    \zeta_j(t)=\frac{1}{2i } \int_0^t \frac{\bar{Z}_j(s)dZ_j(s)-Z_j(s) d\bar{Z}_j(s)}{|Z_j(s)|^2} ,
    \]
    and $\eta$ is the simultaneous winding from Corollary \ref{limit simul wind} for $\mu=(1,\dots,1)$. We therefore conclude from Corollary \ref{limit simul wind} that the following convergence
    \begin{align*}
   & \frac{1}{2i\int_0^t \frac{ds}{|Z(s)|^2}}  \left( \int_0^t \frac{\bar{Z}_1(s)dZ_1(s)-Z_1(s) d\bar{Z}_1(s)}{|Z_1(s)|^2},\dots, \int_0^t \frac{\bar{Z}_n(s)dZ_n(s)-Z_n(s) d\bar{Z}_n(s)}{|Z_n(s)|^2}\right) \\
   \to & \left( C^1_{n-1},\dots,C^n_{n-1}\right)
 \end{align*}
 holds in distribution when $t \to +\infty$, where  $C^1_{n-1},\dots,C^n_{n-1}$ are independent Cauchy random variables with parameter $n-1$. On the other hand, from the Birkhoff-Khinchin ergodic theorem one can check that the following convergence holds almost surely, see \cite[Exercise 3.20, Page 430]{MR1725357}:
 \[
 \lim_{t\to +\infty} \frac{1}{\ln t}\int_0^t \frac{ds}{|Z(s)|^2}=\frac{1}{2(n-1)}.
 \]
 Therefore, one obtains that the following convergence
    \begin{align*}
   & \frac{1}{i \ln t }  \left( \int_0^t \frac{\bar{Z}_1(s)dZ_1(s)-Z_1(s) d\bar{Z}_1(s)}{|Z_1(s)|^2},\dots, \int_0^t \frac{\bar{Z}_n(s)dZ_n(s)-Z_n(s) d\bar{Z}_n(s)}{|Z_n(s)|^2}\right) \\
   \to & \left( C^1_{1},\dots,C^n_{1}\right)
 \end{align*}
 holds in distribution when $t \to +\infty$, where  $C^1_{1},\dots,C^n_{1}$ are independent Cauchy random variables with parameter $1$. This recovers Spitzer's theorem \cite[Theorem 4.1, Page 430]{MR1725357}.

To finish, we point out that the distribution of the simultaneous winding numbers of Brownian loops on spheres can, in principle,  be computed explicitly as a generalization of \cite[Proposition 2.9, Theorem 2.11]{baudoin2024lawindexbrownianloops}. Indeed, we have the following lemma which generalizes \cite[Lemma 6.1]{Yor-lacet} to the sphere setting.

\begin{lemma}
    Let $\Lambda=(\Lambda_1,\dots,\Lambda_n)$ be a random variable taking values on the $(n-1)$-simplex $\mathcal{T}_n $ and $\Theta=(\Theta_1,\dots,\Theta_n)$ be a random variable in $\mathbb{R}^n$. Consider the $\mathbb{S}^{2n-1}$ valued random variable
    \[
    X=(\sqrt{\Lambda_1} e^{i\Theta_1},\dots,\sqrt{\Lambda_n} e^{i\Theta_n}).
    \]
    Assume that $\Lambda$ has a density $U$ with respect to the normalized uniform volume measure $\mu_{\mathcal{T}_{n}} $ of $\mathcal{T}_n $ and that the distribution of $\Theta$ conditionally to $\Lambda=\lambda$ is given by $d\mathbb{P}_{\Theta \mid \Lambda=\lambda}=V(\lambda,\theta) d\theta $, where $d\theta$ denotes here the Lebesgue measure on $\mathbb{R}^n$. Finally, assume that $X$ has a density $p$ with respect to the normalized volume measure of $\mathbb{S}^{2n-1}$. Then, the distribution of $\Theta$ conditioned on $X=(\sqrt{\lambda_1} e^{i\theta_1},\dots, \sqrt{\lambda}_n e^{i \theta_n})$ is given by
    \[
    \mathbb{P}\left( \Theta = \theta +2\pi k \mid X =(\sqrt{\lambda_1} e^{i\theta_1},\dots, \sqrt{\lambda}_n e^{i \theta_n}) \right)= C_n\frac{U(\lambda)V( \lambda, \theta +2\pi k)}{p(\sqrt{\lambda_1} e^{i\theta_1},\dots, \sqrt{\lambda}_n e^{i \theta_n}) }, \quad k \in \mathbb{Z}^n
    \]
    for $\lambda \in \mathcal{T}_{n}$ and $\theta \in [0,2\pi)^n$, where $C_n$ is a normalization constant.
\end{lemma}

\begin{proof}
The statement easily follows from the formula
\[
p(\sqrt{\lambda_1} e^{i\theta_1},\dots, \sqrt{\lambda}_n e^{i \theta_n})= C_nU(\lambda)  \sum_{k \in \mathbb Z^n} V( \lambda, \theta +2\pi k),
\]
which can be deduced from the following argument. Let $f$ be a bounded Borel function. On one hand, 
\begin{align*}
\mathbb{E}\left( f(X) \right)&=\mathbb{E}\left( f(\sqrt{\Lambda_1} e^{i\Theta_1},\dots,\sqrt{\Lambda_n} e^{i\Theta_n}) \right) \\
 &=\int_{\mathcal{T}_{n}} \mathbb{E}\left( f(\sqrt{\lambda_1} e^{i\Theta_1},\dots,\sqrt{\lambda_n} e^{i\Theta_n})  \mid \Lambda=\lambda \right) U(\lambda) d\mu_{\mathcal{T}_{n}} (\lambda) \\
 &=\int_{\mathbb{R}^n}\int_{\mathcal{T}_{n}}  f(\sqrt{\lambda_1} e^{i\theta_1},\dots, \sqrt{\lambda}_n e^{i \theta_n})  V(\lambda,\theta) U(\lambda) d\mu_{\mathcal{T}_{n}} (\lambda) d\theta \\
 &=\int_{[0,2\pi]^n}\int_{\mathcal{T}_{n}}  f(\sqrt{\lambda_1} e^{i\theta_1},\dots, \sqrt{\lambda}_n e^{i \theta_n}) \left( \sum_{k \in \mathbb{Z}^n} V(\lambda,\theta+2\pi k) \right) U(\lambda) d\mu_{\mathcal{T}_{n}} (\lambda)d\theta ,
\end{align*}
while on the other hand, one has
\begin{align*}
\mathbb{E}\left( f(X) \right)=\int_{\mathbb{S}^{2n-1}} f(x) p(x) d\mu_{\mathbb{S}^{2n-1}} (x).
\end{align*}
Performing in the last integral the change of variable $x=(\sqrt{\lambda_1} e^{i\theta_1},\dots, \sqrt{\lambda}_n e^{i \theta_n})$ with $\lambda \in \mathcal{T}_{n}$  and $\theta \in [0,2\pi)^n$ yields the expected result after a straightforward computation of the Jacobian.
\end{proof}

Combining the above lemma with Theorems \ref{Representation BM sphere}  and \ref{carac area flag}  yields a formula for  $$\mathbb{P}\left(\eta(t)=\eta(0)+2\pi k \mid X(t)=X(0)\right) ,$$ where $(X(t))_{t\geq 0}=(X^1(t),\dots,X^n(t))_{t\geq 0}$ is a Brownian motion on $\mathbb{S}^{2n-1}$ such that $X^j(0) \neq 0$, $1 \le j \le n$ and $(\eta^1(t),\dots,\eta^n(t))_{t \ge 0}$ is a continuous stochastic process such that
    \[
    X^j(t)=| X^j(t) | e^{i\eta^j(t) }
    \]
    with $\eta(0) \in [0,2\pi)^n $. We restrain to explicitly write the formula here, since it involves inverting the conditional Fourier transform of Theorem \ref{carac area flag}, which is difficult to handle.

\bibliographystyle{plain}
\bibliography{reference.bib}

\vspace{5pt}
\noindent
\begin{minipage}{\textwidth}
    \small
    \textbf{Fabrice Baudoin:} \\
    Department of Mathematics, Aarhus University \\
    Email: fbaudoin@math.au.dk
\end{minipage}

\vspace{10pt} 

\noindent
\begin{minipage}{\textwidth}
    \small
    \textbf{Nizar Demni:} \\
    Department of Mathematics, NYU Abu Dhabi \\
    Email: nd2889@nyu.edu
\end{minipage}

\vspace{10pt}

\noindent
\begin{minipage}{\textwidth}
    \small
    \textbf{Teije Kuijper:} \\
    Department of Mathematics, Aarhus University \\
    Email: t.kuijper@math.au.dk
\end{minipage}

\vspace{10pt}

\noindent
\begin{minipage}{\textwidth}
    \small
    \textbf{Jing Wang:} \\
    Department of Mathematics, Purdue University \\
    Email: jingwang@purdue.edu
\end{minipage}

\end{document}